\documentclass[10pt]{article}
\usepackage[utf8]{inputenc}
\usepackage{latexsym}
\usepackage[english]{babel}
\usepackage{geometry}
\usepackage{booktabs}
\usepackage{array}
\usepackage{paralist}
\usepackage{verbatim}
\usepackage{subfig}
\usepackage{amsmath}
\usepackage{amsfonts}
\usepackage{hyperref}
\usepackage{bbm}
\usepackage{amsthm}
\usepackage{amssymb}
\usepackage{mathtools}
\usepackage{bbold}

\makeatletter
\newtheorem*{rep@theorem}{\rep@title}
\newcommand{\newreptheorem}[2]{%
\newenvironment{rep#1}[1]{%
 \def\rep@title{#2 \ref{##1}}%
 \begin{rep@theorem}}%
 {\end{rep@theorem}}}
\makeatother

\newtheorem{thm}{Theorem}[section]
\newtheorem{cor}[thm]{Corollary}
\newtheorem{lem}[thm]{Lemma}
\newtheorem{prop}[thm]{Proposition}
\newtheorem{defn}[thm]{Definition}
\newtheorem{rem}[thm]{Remark}

\newtheorem*{cor*}{Corollary}

\newtheorem{theorem}{Theorem}
\newreptheorem{theorem}{Theorem}

\newreptheorem{proposition}{Proposition}

\newreptheorem{corollary}{Corollary}

\usepackage{fancyhdr}
\usepackage{sectsty}
\allsectionsfont{\sffamily\mdseries\upshape}

\usepackage[nottoc,notlof,notlot]{tocbibind}
\usepackage[titles,subfigure]{tocloft}

\title{The strong topology of $\omega$-plurisubharmonic functions}
\author{Antonio Trusiani\footnote{\href{mailto:antonio.trusiani91@gmail.com}{antonio.trusiani91@gmail.com};}}
\date{}

\begin{document}
\maketitle
\begin{abstract}
On $(X,\omega)$ compact Kähler manifold, given a model type envelope $\psi\in PSH(X,\omega)$ (i.e. a singularity type) we prove that the Monge-Ampère operator is a homeomorphism between the set of $\psi$-relative finite energy potentials and the set of $\psi$-relative finite energy measures endowed with their strong topologies given as the coarsest refinements of the weak topologies such that the relative energies become continuous. Moreover, given a totally ordered family $\mathcal{A}$ of model type envelopes with positive total mass representing different singularities types, the sets $X_{\mathcal{A}}, Y_{\mathcal{A}}$ given respectively as the union of all $\psi$-relative finite energy potentials and of all $\psi$-relative finite energy measures varying $\psi\in\overline{\mathcal{A}}$ have two natural strong topologies which extends the strong topologies on each component of the unions. We show that the Monge-Ampère operator produces a homeomorphism between $X_{\mathcal{A}}$ and $Y_{\mathcal{A}}$.\\
As an application we also prove the strong stability of a sequence of solutions of complex Monge-Ampère equations when the measures have uniformly $L^{p}$-bounded densities for $p>1$ and the prescribed singularities are totally ordered.
\end{abstract}
\vspace{5pt}
{\small \textbf{Keywords:} Complex Monge-Ampère equations, compact Kähler manifolds, quasi-psh functions.}
\section{Introduction}
Let $(X,\omega)$ be a compact Kähler manifold where $\omega$ is a fixed Kähler form, and let $\mathcal{H}_{\omega}$ denote the set of all Kähler potentials, i.e. all $\varphi \in C^{\infty}$ such that $\omega+dd^{c}\varphi$ is a Kähler form, the pioneering work of Yau (\cite{Yau78}) shows that the Monge-Ampère operator
\begin{equation}
\label{eqn:Yau}
MA_{\omega}:\mathcal{H}_{\omega,norm}\longrightarrow \Big\{dV\, \mbox{volume form}\, : \, \int_{X}dV=\int_{X}\omega^{n}\Big\},
\end{equation}
$MA_{\omega}(\varphi):=(\omega+dd^{c}\varphi)^{n}$ is a bijection, where for any subset $A\subset PSH(X,\omega)$ of all $\omega$-plurisubharmonic functions we use the notation $A_{norm}:=\{u\in A\, : \, \sup_{X}u=0\}$. Note that the assumption on the total mass of the volume forms in (\ref{eqn:Yau}) is necessary since $\mathcal{H}_{\omega,norm}$ represents all Kähler forms in the cohomology class $\{\omega\}$ and the quantity $\int_{X}\omega^{n}$ is cohomological.\\
In \cite{GZ07} the authors extended the Monge-Ampère operator using the \emph{non-pluripolar product} (as denominated successively in \cite{BEGZ10}) and the bijection (\ref{eqn:Yau}) to
\begin{equation}
\label{eqn:Bij1}
MA_{\omega}:\mathcal{E}_{norm}(X,\omega)\longrightarrow \Big\{\mu\, \mbox{non-pluripolar positive measure}\, :\, \mu(X)=\int_{X}\omega^{n}\Big\}
\end{equation}
where $\mathcal{E}(X,\omega):=\{u\in PSH(X,\omega)\, : \, \int_{X}MA_{\omega}(u)=\int_{X}MA_{\omega}(0)\}$ is the set of all $\omega$-psh functions with full Monge-Ampère mass.\\
The set $PSH(X,\omega)$ is naturally endowed with the $L^{1}$-topology which we will call \emph{weak}, but the Monge-Ampère operator in (\ref{eqn:Bij1}) is not continuous even if the set of measures is endowed with the weak topology. Thus in \cite{BBEGZ16}, setting $V_{0}:=\int_{X}MA_{\omega}(0)$, two strong topologies were respectively introduced for
\begin{gather*}
\mathcal{E}^{1}(X,\omega):=\{u\in\mathcal{E}(X,\omega)\, : \, E(u)>-\infty\}\\
\mathcal{M}^{1}(X,\omega):=\Big\{V_{0}\mu \, : \,\mu\, \mbox{is a probability measure satisfying} \,  E^{*}(\mu)<+\infty\Big\}
\end{gather*}
as the coarsest refinements of the weak topologies such that respectively the Monge-Ampère energy $E(u)$ (\cite{Aub84}, \cite{BB08}, \cite{BEGZ10}) and the energy for probability measures $E^{*}$ (\cite{BBGZ09}, \cite{BBEGZ16}) becomes continuous. The map
\begin{equation}
\label{eqn:Homeo0}
MA_{\omega}:\big(\mathcal{E}^{1}_{norm}(X,\omega),strong\big)\longrightarrow \big(\mathcal{M}^{1}(X,\omega),strong\big)
\end{equation}
is then a homeomorphism. Later Darvas (\cite{Dar14}) showed that actually $\big(\mathcal{E}^{1}(X,\omega), strong\big)$ coincides with the metric closure of $\mathcal{H}_{\omega}$ endowed with the Finsler metric $|f|_{1,\varphi}:=\int_{X}|f|MA_{\omega}(\varphi)$, $\varphi\in \mathcal{H}_{\omega}$, $f\in T_{\varphi}\mathcal{H}_{\omega}\simeq C^{\infty}(X)$ and associated distance
$$
d(u,v):=E(u)+E(v)-2E\big(P_{\omega}(u,v)\big)
$$
where $P_{\omega}(u,v)$ is the rooftop envelope given basically as the largest $\omega$-psh function bounded above by $\min(u,v)$ (\cite{RWN14}). This metric topology has played an important role in the last decade to characterize the existence of special metrics (\cite{DR15}, \cite{BDL16}, \cite{CC17}, \cite{CC18a}, \cite{CC18b}).\\

It is also important and natural to solve complex Monge-Ampère equations requiring that the solutions have some prescribed behavior, for instance along a divisor.\\
We first need to recall that on $PSH(X,\omega) $ there is a natural partial order $\preccurlyeq$ given as $u\preccurlyeq v$ if $u\leq v+ O(1)$, and the total mass through the Monge-Ampère operator respects such partial order, i.e. $V_{u}:=\int_{X}MA_{\omega}(u)\leq V_{v}$ if $u\preccurlyeq v$ (\cite{BEGZ10}, \cite{WN17}). Thus in \cite{DDNL17b} the authors introduced the $\psi$\emph{-relative} analogs of the sets $\mathcal{E}(X,\omega)$, $\mathcal{E}^{1}(X,\omega)$ for $\psi\in PSH(X,\omega)$ fixed as
\begin{gather*}
\mathcal{E}(X,\omega,\psi):=\{u\in PSH(X,\omega)\, : \, u\preccurlyeq \psi\, \mbox{and}\, V_{u}=V_{v}\}\\
\mathcal{E}^{1}(X,\omega,\psi):=\{u\in\mathcal{E}(X,\omega,\psi)\, : \, E_{\psi}(u)>-\infty\}
\end{gather*}
where $E_{\psi}$ is the $\psi$-relative energy, and they proved that
\begin{equation}
\label{eqn:Bij2}
MA_{\omega} : \mathcal{E}_{norm}(X,\omega,\psi)\longrightarrow \Big\{\mu \, \mbox{non-pluripolar positive measure}\, : \, \mu(X)=V_{\psi}\Big\}
\end{equation}
is a bijection if and only if $\psi$, up to a bounded function, is a \emph{model type envelope}, i.e. $\psi=(\lim_{C\to +\infty}P(\psi+C,0)\big)^{*}$, satisfying $V_{\psi}>0$ (the star is for the upper semicontinuous regularization). There are plenty of these functions, for instance to any $\omega$-psh function $\psi$ with analytic singularities is associated a unique model type envelope. We denote by $\mathcal{M}$ the set of all model type envelopes and with $\mathcal{M}^{+}$ those elements $\psi$ such that $V_{\psi}>0$.\\
Letting $\psi\in \mathcal{M}^{+}$, in \cite{Tru19}, we proved that $\mathcal{E}^{1}(X,\omega,\psi)$ can be endowed with a natural metric topology given by the complete distance $d(u,v):=E_{\psi}(u)+E_{\psi}(v)-2E_{\psi}\big(P_{\omega}(u,v)\big)$. \\
Analogously to $E^{*}$, we introduce in section \S \ref{sec:Varia} a natural $\psi$-relative energy for probability measures $E^{*}_{\psi}$, thus the set
$$
\mathcal{M}^{1}(X,\omega,\psi):=\{V_{\psi}\mu\,:\, \mu\, \mbox{is a probability measure satisfying} \, E^{*}_{\psi}(\mu)<+\infty\}
$$
can be endowed with its strong topology given as the coarsest refinement of the weak topology such that $E^{*}_{\psi}$ becomes continuous.
\begin{theorem}
\label{thmA}
Let $\psi\in\mathcal{M}^{+}$. Then
\begin{equation}
\label{eqn:HomeoP}
MA_{\omega}:\big(\mathcal{E}^{1}_{norm}(X,\omega,\psi),d\big)\to \big(\mathcal{M}^{1}(X,\omega,\psi),strong\big)
\end{equation}
is a homeomorphism.
\end{theorem}

Then it is natural to wonder if one can extend the bijections (\ref{eqn:Bij1}), (\ref{eqn:Bij2}) to bigger subsets of $PSH(X,\omega)$. \\
Given $\psi_{1},\psi_{2}\in\mathcal{M}^{+}$ such that $\psi_{1}\neq \psi_{2}$ the sets $\mathcal{E}(X,\omega,\psi_{1})$, $\mathcal{E}(X,\omega,\psi_{2})$ are disjoint (Theorem $1.3$ \cite{DDNL17b} quoted below as Theorem \ref{thm:ClassE}) but it may happen that $V_{\psi_{1}}=V_{\psi_{2}}$. So in these situations, at least one of $\mathcal{E}^{1}_{norm}(X,\omega,\psi_{1}), \mathcal{E}^{1}_{norm}(X,\omega,\psi_{2})$ must be ruled out to extend $(\ref{eqn:Bij2})$. However, given a totally ordered family $\mathcal{A}\subset \mathcal{M}^{+}$ of model type envelopes, the map $\mathcal{A}\ni\psi\to V_{\psi}$ is injective (again by Theorem $1.3$ \cite{DDNL17b}), i.e.
$$
MA_{\omega}: \bigsqcup_{\psi\in\mathcal{A}} \mathcal{E}_{norm}(X,\omega,\psi)\longrightarrow \Big\{\mu\, \mbox{non-pluripolar positive measure}\,: \, \mu(X)=V_{\psi}\, \mbox{for}\, \psi\in\mathcal{A}\Big\}  
$$
is a bijection. \\
In \cite{Tru19} we introduced a complete distance $d_{\mathcal{A}}$ on
$$
X_{\mathcal{A}}:=\bigsqcup_{\psi\in\overline{\mathcal{A}}}\mathcal{E}^{1}(X,\omega,\psi)
$$
where $\overline{\mathcal{A}}\subset \mathcal{M}$ is the weak closure of $\mathcal{A}$ and where we identify $\mathcal{E}^{1}(X,\omega,\psi_{\min})$ with a point $P_{\psi_{\min}}$ if $\psi_{\min}\in\mathcal{M}\setminus \mathcal{M}^{+}$ (since in this case $E_{\psi}\equiv 0$, see Remark \ref{rem:0Mass}). Here $\psi_{\min}$ is given as the smallest element in $\overline{\mathcal{A}}$, observing that the Monge-Ampère operator $MA_{\omega}:\overline{\mathcal{A}}\to MA_{\omega}(\overline{\mathcal{A}})$ is a homeomorphism when the range is endowed with the weak topology (Lemma \ref{lem:HomeoV}). We call strong topology on $X_{\mathcal{A}}$ the metric topology given by $d_{\mathcal{A}}$ since $d_{\mathcal{A}|\mathcal{E}^{1}(X,\omega,\psi)\times \mathcal{E}^{1}(X,\omega,\psi)}=d$. The precise definition of $d_{\mathcal{A}}$ is quite technical (in section \S \ref{sec:Pre} we will recall many of its properties) but the strong topology is natural since it is the coarsest refinement of the weak topology such that $E_{\cdot}(\cdot)$ becomes continuous as Theorem \ref{thm:OldPropA} shows. In particular the strong topology is independent on the set $\mathcal{A}$ chosen.\\
Also the set
$$
Y_{\mathcal{A}}:=\bigsqcup_{\psi\in \overline{\mathcal{A}}}\mathcal{M}^{1}(X,\omega,\psi)
$$
has a natural strong topology given as the coarsest refinement of the weak topology such that $E_{\cdot}^{*}(\cdot)$ becomes continuous.
\begin{theorem}
\label{thmB}
The Monge-Ampère map
$$
MA_{\omega}:\big(X_{\mathcal{A},norm},d_{\mathcal{A}}\big)\to (Y_{\mathcal{A}},strong)
$$
is a homeomorphism.
\end{theorem}
Obviously in Theorem \ref{thmB} we define $MA_{\omega}(P_{\psi_{\min}}):=0$ if $V_{\psi_{\min}}=0$.\\
Note that by Hartogs' Lemma and Theorem \ref{thm:OldPropA} the metric subspace $X_{\mathcal{A},norm}$ is complete and it represents the set of all closed and positive $(1,1)$-currents $T=\omega+dd^{c}u$ such that $u\in X_{\mathcal{A}}$, where $P_{\psi_{\min}}$ encases all currents whose potentials $u$ are more singular than $\psi_{\min}$ if $V_{\psi_{\min}}=0$.\\

Finally, as an application of Theorem \ref{thmB} we study an example of the stability of solutions of complex Monge-Ampère equations. Other important situations will be dealt in a future work.
\begin{theorem}
\label{thmC}
Let $\mathcal{A}:=\{\psi_{k}\}_{k\in\mathbbm{N}}\subset \mathcal{M}^{+}$ be totally ordered, and let $\{f_{k}\}_{k\in\mathbbm{N}}\subset L^{1}\setminus\{0\}$ a sequence of non-negative functions such that $f_{k}\to f\in L^{1}\setminus\{0\}$ and such that $\int_{X}f_{k}\omega^{n}=V_{\psi_{k}}$ for any $k\in\mathbbm{N}$. Assume also that there exists $p>1$ such that $||f_{k}||_{L^{p}},||f||_{L^{p}}$ are uniformly bounded. Then $\psi_{k}\to \psi\in\mathcal{M}^{+}$ weakly, the sequence $\{u_{k}\}_{k\in\mathbbm{N}}$ of solutions of
\begin{equation}
\label{eqn:LP}
\begin{cases}
MA_{\omega}(u_{k})=f_{k}\omega^{n}\\
u_{k}\in\mathcal{E}^{1}_{norm}(X,\omega,\psi_{k})
\end{cases}
\end{equation}
converges strongly to $u\in X_{\mathcal{A}}$ (i.e. $d_{\mathcal{A}}(u_{k},u)\to 0$), which is the unique solution of
\begin{equation*}
\begin{cases}
MA_{\omega}(u)=f\omega^{n}\\
u\in\mathcal{E}^{1}_{norm}(X,\omega,\psi).
\end{cases}
\end{equation*}
In particular $u_{k}\to u$ in capacity.
\end{theorem}
The existence of the solutions of (\ref{eqn:LP}) follows by Theorem $A$ in \cite{DDNL18b}, while the fact that the strong convergence implies the convergence in capacity is our Theorem \ref{thm:OldPropB}. Note also that the convergence in capacity of Theorem \ref{thmC} was already obtained in \cite{DDNL19} (see Remark \ref{rem:Ele}).
\subsection{Structure of the paper}
Section \S \ref{sec:Pre} is dedicated to introduce some preliminaries, and in particular all necessary results presented in \cite{Tru19}. In section \S \ref{sec:Tools} we extend some known uniform estimates for $\mathcal{E}^{1}(X,\omega)$ to the relative setting, and we prove the key upper-semicontinuity of the relative energy functional $E_{\cdot}(\cdot)$ in $X_{\mathcal{A}}$. Section \S \ref{sec:ActionMeasure} regards the properties of the action of measures on $PSH(X,\omega)$ and in particular their continuity. Then Section \S \ref{sec:Varia} is dedicated to prove Theorem \ref{thmA}. We use a variational approach to show the bijection, then we need some further important properties of the strong topology on $\mathcal{E}^{1}(X,\omega,\psi)$ to conclude the proof. Section \S \ref{sec:Strong} is the heart of the article where we extends the results proved in the previous section to $X_{\mathcal{A}}$ and we present our main Theorem \ref{thmB}. Finally in the last Section \S \ref{sec:CMAE} we show Theorem \ref{thmC}.
\subsection{Future developments}
As said above, in a future work we will present some strong stability results of more general solutions of complex Monge-Ampère equations with prescribed singularities than Theorem \ref{thmC}, starting the study of a kind of \emph{continuity method} when also the singularities will vary. As an application we will study the existence of (log) Kähler-Einstein metrics with prescribed singularities with a particular focus on the relationships among them varying the singularities.
\subsection{Acknowledgments}
I want to thank David Witt Nyström and Stefano Trapani for their suggestions and comments. I am also grateful to Hoang-Chinh Lu to have pointed me out a minor mistake in the previous version.
\section{Preliminaries}
\label{sec:Pre}
We recall that given $(X,\omega)$ a Kähler complex compact manifold, the set $PSH(X,\omega)$ is the set of all $\omega$-plurisubharmonic functions ($\omega$-psh), i.e. all $u\in L^{1}$ given locally as sum of a smooth function and of a plurisubharmonic function such that $\omega+dd^{c}u\geq 0$ as $(1,1)$-current. Here $d^{c}:=\frac{i}{2\pi}(\bar{\partial}-\partial)$ so that $dd^{c}=\frac{i}{\pi}\partial \bar{\partial}$.
For any couple of $\omega$-psh functions $u,v$ the function
$$
P_{\omega}[u](v):=\Big(\lim_{C\to \infty}P_{\omega}(u+C,v)\Big)^{*}=\Big(\sup\{w\in PSH(X,\omega)\, : \, w\preccurlyeq u, w\leq v\}\Big)^{*}
$$
is $\omega$-psh where the star is for the upper semicontinuous regularization and $P_{\omega}(u,v):=\big(\sup\{w\in PSH(X,\omega)\, :\, w\leq \min(u,v)\}\big)^{*}$. Then the set of all model type envelopes is defined as
$$
\mathcal{M}:=\{\psi\in PSH(X,\omega)\, : \, \psi=P_{\omega}[\psi](0)\}.
$$
We also recall that $\mathcal{M}^{+}$ denotes the elements $\psi\in\mathcal{M}$ such that $V_{\psi}>0$ where, as said in the Introduction, $V_{\psi}:=\int_{X}MA_{\omega}(\psi)$.\\
The class of $\psi$-relative full mass functions $\mathcal{E}(X,\omega,\psi)$ complies the following characterization.
\begin{thm}[Theorem $1.3$, \cite{DDNL17b}]
\label{thm:ClassE}
Suppose $v\in PSH(X,\omega)$ such that $V_{v}>0$ and $u\in PSH(X,\omega)$ more singular than $v$. The followings are equivalent:
\begin{itemize}
\item[(i)] $u\in\mathcal{E}(X,\omega,v)$;
\item[(ii)] $P_{\omega}[u](v)=v$;
\item[(iii)] $P_{\omega}[u](0)=P_{\omega}[v](0)$.
\end{itemize}
\end{thm}
The clear inclusion $\mathcal{E}(X,\omega,v)\subset \mathcal{E}(X,\omega,P_{\omega}[v](0))$ may be strict, and it seems more natural in many cases to consider only functions $\psi\in \mathcal{M}$. For instance as showed in \cite{DDNL17b} $\psi$ being a model type envelope is a necessary assumption to make the equation
$$
\begin{cases}
MA_{\omega}(u)=\mu\\
u\in\mathcal{E}(X,\omega,\psi)
\end{cases}
$$
always solvable where $\mu$ is a non-pluripolar measure such that $\mu(X)=V_{\psi}$. It is also worth to recall that there are plenty of elements in $\mathcal{M}$ since $P_{\omega}[P_{\omega}[\psi]]=P_{\omega}[\psi]$ for any $\psi\in PSH(X,\omega)$ with $\int_{X}MA_{\omega}(\psi)>0$ (Theorem $3.12$, \cite{DDNL17b}). Indeed $v\to P_{\omega}[v]$ may be thought as a projection from the set of negative $\omega$-psh functions with positive Monge-Ampère mass to $\mathcal{M}^{+}$.\\
We also retrieve the following useful result.
\begin{thm}[Theorem 3.8, \cite{DDNL17b}]
\label{thm:3.8}
Let $u,\psi\in PSH(X,\omega)$ such that $u\succcurlyeq \psi$. Then
$$
MA_{\omega}\big(P_{\omega}[\psi](u)\big)\leq \mathbbm{1}_{\{P_{\omega}[\psi](u)=u\}}MA_{\omega}(u).
$$
In particular if $\psi\in\mathcal{M}$ then $MA_{\omega}(\psi)\leq \mathbbm{1}_{\{\psi=0\}}MA_{\omega}(0)$.
\end{thm}
Note also that in Theorem \ref{thm:3.8} the equality holds if $u$ is continuous with bounded distributional laplacian with respect to $\omega$ as a consequence of \cite{DNT19}. In particular $MA_{\omega}(\psi)=\mathbbm{1}_{\{\psi=0\}}MA_{\omega}(0)$ for any $\psi\in\mathcal{M}$.
\subsection{The metric space $\big(\mathcal{E}^{1}(X,\omega,\psi),d\big)$.}
\label{ssec:Preli21}
In this subsection we assume $\psi\in\mathcal{M}^{+}$ where $\mathcal{M}^{+}:=\{\psi\in\mathcal{M}\, :\, V_{\psi}>0\}$.\\
As in \cite{DDNL17b} we also denote by $PSH(X,\omega,\psi)$ the set of all $\omega$-psh functions which are more singular than $\psi$, and we recall that a function $u\in PSH(X,\omega,\psi)$ has $\psi$-\emph{relative minimal singularities} if $|u-\psi|$ is globally bounded on $X$. We also use the notation 
$$
MA_{\omega}(u_{1}^{j_{1}},\dots,u_{l}^{j_{l}}):=(\omega+dd^{c}u_{1})^{j_{1}}\wedge \cdots\wedge (\omega+dd^{c}u_{l})^{j_{l}}
$$
for $u_{1},\dots,u_{l}\in PSH(X,\omega)$ where $j_{1},\dots,j_{l}\in\mathbbm{N}$ such that $j_{1}+\cdots+j_{l}=n$.
\begin{defn}[Section \S 4.2, \cite{DDNL17b}]
The $\psi$-\emph{relative energy functional} $E_{\psi}:PSH(X,\omega,\psi)\to \mathbbm{R}\cup \{-\infty\}$ is defined as
$$
E_{\psi}(u):=\frac{1}{n+1}\sum_{j=0}^{n}\int_{X}(u-\psi) MA_{\omega}(u^{j},\psi^{n-j})
$$
if $u$ has $\psi$-relative minimal singularities, and as
$$
E_{\psi}(u):=\inf\{E_{\psi}(v)\, :\, v\in\mathcal{E}(X,\omega,\psi)\, \mbox{with} \, \psi\mbox{-relative minimal singularities}, v\geq u\}
$$
otherwise. The subset $\mathcal{E}^{1}(X,\omega,\psi)\subset \mathcal{E}(X,\omega,\psi)$ is defined as
$$
\mathcal{E}^{1}(X,\omega,\psi):=\{u\in\mathcal{E}(X,\omega,\psi)\,:\, E_{\psi}(u)>-\infty\}.
$$
\end{defn}
When $\psi=0$ the $\psi-$relative energy functional is the \emph{Aubin-Mabuchi energy functional}, also called \emph{Monge-Ampére energy} (see \cite{Aub84},\cite{Mab86}).
\begin{prop}
\label{prop:PropertiesE}
The following properties hold:
\begin{itemize}
\item[(i)] $E_{\psi}$ is non decreasing \emph{(Theorem $4.10$, \cite{DDNL17b})};
\item[(ii)] $E_{\psi}(u)=\lim_{j\to \infty}E_{\psi}\big(\max(u,\psi-j)\big)$ \emph{(Lemma $4.12$, \cite{DDNL17b})};
\item[(iii)] $E_{\psi}$ is continuous along decreasing sequences \emph{(Lemma $4.14$, \cite{DDNL17b})};
\item[(iv)] $E_{\psi}$ is concave along affine curves \emph{(Theorem $4.10$, Corollary $4.16$, \cite{DDNL17b})};
\item[(v)] $u\in \mathcal{E}^{1}(X,\omega,\psi)$ if and only if $u\in\mathcal{E}(X,\omega,\psi)$ and $\int_{X}(u-\psi)MA_{\omega}(u)>-\infty$ \emph{(Lemma $4.13$, \cite{DDNL17b})};
\item[(vi)] $E_{\psi}(u)\geq \limsup_{k\to \infty}E_{\psi}(u_{k})$ if $u_{k},u\in\mathcal{E}^{1}(X,\omega,\psi)$ and $u_{k}\to u$ with respect to the weak topology \emph{(Proposition $4.19$, \cite{DDNL17b})};
\item[(vii)] letting $u\in\mathcal{E}^{1}(X,\omega,\psi)$, $\chi\in \mathcal{C}^{0}(X)$ and $u_{t}:=\sup\{v\in PSH(X,\omega) \, v\leq u+t\chi\}^{*}$ for any $t>0$, then $t\to E_{\psi}(u_{t})$ is differentiable and its derivative is given by
$$
\frac{d}{dt}E_{\psi}(u_{t})=\int_{X}\chi MA_{\omega}(u_{t})
$$
\emph{(Proposition $4.20$, \cite{DDNL17b})};
\item[(viii)] if $u,v\in \mathcal{E}^{1}(X,\omega,\psi)$ then
$$
E_{\psi}(u)-E_{\psi}(v)=\frac{1}{n+1}\sum_{j=0}^{n}\int_{X}(u-v)MA_{\omega}(u^{j},v^{n-j})
$$
and the function $\mathbbm{N}\ni j\to \int_{X}(u-v)MA_{\omega}(u^{j},v^{n-j})$ is decreasing. In particular
$$
\int_{X}(u-v)MA_{\omega}(u)\leq E_{\psi}(u)-E_{\psi}(v)\leq \int_{X}(u-v)MA_{\omega}(v)
$$
\emph{(Theorem $4.10$, \cite{DDNL17b})};
\item[(ix)] if $u\leq v$ then $E_{\psi}(u)-E_{\psi}(v)\leq \frac{1}{n+1}\int_{X}(u-v)MA_{\omega}(u)$ \emph{(Theorem $4.10$, \cite{DDNL17b})}.
\end{itemize}
\end{prop}
\begin{rem}
\emph{All the properties of Proposition \ref{prop:PropertiesE} are showed in \cite{DDNL17b} assuming $\psi$ having \emph{small unbounded locus}, but the general integration by parts formula proved in \cite{X19} and Proposition $2.7$ in \cite{Tru19} allow to extend these properties to the general case as desribed in Remark $2.10$ in \cite{Tru19}.}
\end{rem} 
Recalling that for any $u,v\in\mathcal{E}^{1}(X,\omega,\psi)$ the function $P_{\omega}(u,v)=\sup\{w\in PSH(X,\omega)\, : \, w\leq \min(u,v)\}^{*}$ belongs to $\mathcal{E}^{1}(X,\omega,\psi)$ (see Proposition $2.13.$ in \cite{Tru19}), the function $d:\mathcal{E}^{1}(X,\omega,\psi)\times\mathcal{E}^{1}(X,\omega,\psi)\to \mathbbm{R}_{\geq 0}$ defined as
$$
d(u,v)=E_{\psi}(u)+E_{\psi}(v)-2E_{\psi}\big(P_{\omega}(u,v)\big)
$$
assumes finite values. Moreover it is a complete distance as the next result shows.
\begin{thm}[Theorem A, \cite{Tru19}]
$\big(\mathcal{E}^{1}(X,\omega,\psi),d\big)$ is a complete metric space.
\end{thm}
We call \emph{strong topology} on $\mathcal{E}^{1}(X,\omega,\psi)$ the metric topology given by the distance $d$. Note that by construction $d(u_{k},u)\to 0$ as $k\to \infty$ if $u_{k}\searrow u$, and that $d(u,v)=d(u,w)+d(w,v)$ if $u\leq w\leq v$ (see Lemma $3.1$ in \cite{Tru19}).\\
Moreover as a consequence of Proposition \ref{prop:PropertiesE} it follows that for any $C\in\mathbbm{R}_{>0}$ the set
$$
\mathcal{E}_{C}^{1}(X,\omega,\psi):=\{u\in\mathcal{E}^{1}(X,\omega,\psi)\, :\, \sup_{X}u\leq C \, \mbox{and}\, E_{\psi}(u)\geq -C\}
$$
is a weakly compact convex set.
\begin{rem}
\label{rem:0Mass}
\emph{As described in Remark $3.10$ in \cite{Tru19}, if $\psi\in\mathcal{M}\setminus \mathcal{M}^{+}$ then $\mathcal{E}^{1}(X,\omega,\psi)=PSH(X,\omega,\psi)$ since $E_{\psi}\equiv 0$ by definition. In particular $d\equiv 0$ and it is natural to identify $\big(\mathcal{E}^{1}(X,\omega,\psi),d\big)$ with a point $P_{\psi}$. Moreover we recall that $\mathcal{E}^{1}(X,\omega,\psi_{1})\cap \mathcal{E}^{1}(X,\omega,\psi_{2})=\emptyset$ if $\psi_{1},\psi_{2}\in\mathcal{M}$, $\psi_{1}\neq \psi_{2}$ and $V_{\psi_{2}}>0$.}
\end{rem}
\subsection{The space $(X_{\mathcal{A}},d_{\mathcal{A}})$.}
From now on we assume $\mathcal{A}\subset \mathcal{M}^{+}$ to be a totally ordered set of model type envelopes, and we denote by $\overline{\mathcal{A}}$ its closure as subset of $PSH(X,\omega)$ endowed with the weak topology. Note that $\overline{\mathcal{A}}\subset PSH(X,\omega)$ is compact by Lemma $2.6$ in \cite{Tru19}. Indeed we will prove in Lemma \ref{lem:HomeoV} that actually $\overline{\mathcal{A}}$ is homeomorphic to its image through the Monge-Ampère operator $MA_{\omega}$ when the set of measure is endowed with the weak topology, which yields that $\overline{\mathcal{A}}$ is also homeomorphic to a closed set contained in $[0,\int_{X}\omega^{n}]$ through the map $\psi\to V_{\psi}$.
\begin{defn}
We define the set
$$
X_{\mathcal{A}}:=\bigsqcup_{\psi\in\overline{\mathcal{A}}}\mathcal{E}^{1}(X,\omega,\psi)
$$
if $\psi_{\min}:=\inf \mathcal{A}$ satisfies $V_{\psi_{\min}}>0$, and
$$
X_{\mathcal{A}}:=P_{\psi_{\min}}\sqcup \bigsqcup_{\psi'\in\overline{\mathcal{A}},\psi\not=\psi_{\min}}\mathcal{E}^{1}(X,\omega,\psi)
$$
if $V_{\psi_{\min}}=0$, where $P_{\psi_{\min}}$ is a singleton.
\end{defn}
$X_{\mathcal{A}}$ can be endowed with a natural metric structure as section $4$ in \cite{Tru19} shows. 
\begin{thm}[Theorem B, \cite{Tru19}]
$(X_{\mathcal{A}},d_{\mathcal{A}})$ is a complete metric space such that $d_{\mathcal{A}|\mathcal{E}^{1}(X,\omega,\psi)\times\mathcal{E}^{1}(X,\omega,\psi)}=d$ for any $\psi\in\overline{\mathcal{A}}\cap\mathcal{M}^{+}$.
\end{thm}
We call \emph{strong topology} on $X_{\mathcal{A}}$ the metric topology given by the distance $d_{\mathcal{A}}$. Note that the denomination is coherent with that of subsection \ref{ssec:Preli21} since the induced topology on $\mathcal{E}^{1}(X,\omega,\psi)\subset X_{\mathcal{A}}$ coincides with the strong topology given by $d$.\\
We will also need the following contraction property which is the starting point to construct $d_{\mathcal{A}}$.
\begin{prop}[Lemma 4.2., Proposition 4.3., \cite{Tru19}]
\label{prop:PropProie}
Let $\psi_{1},\psi_{2},\psi_{3}\in\mathcal{M}$ such that $\psi_{1}\preccurlyeq\psi_{2}\preccurlyeq \psi_{3}$. Then $P_{\omega}[\psi_{1}]\big(P_{\omega}[\psi_{2}](u)\big)=P_{\omega}[\psi_{1}](u)$ for any $u\in\mathcal{E}^{1}(X,\omega,\psi_{3})$ and $|P_{\omega}[\psi_{1}](u)-\psi_{1}|\leq C$ if $|u-\psi_{3}|\leq C$. Moreover the map
$$
P_{\omega}[\psi_{1}](\cdot): \mathcal{E}^{1}(X,\omega,\psi_{2})\to PSH(X,\omega,\psi_{1})
$$
has image in $\mathcal{E}^{1}(X,\omega,\psi_{1})$ and it is a Lipschitz map of constant $1$ when the sets $\mathcal{E}^{1}(X,\omega,\psi_{i})$, $i=1,2$, are endowed with the $d$ distances, i.e.
$$
d\big(P_{\omega}[\psi_{1}](u), P_{\omega}[\psi_{1}](v)\big)\leq d(u,v)
$$
for any $u,v\in\mathcal{E}^{1}(X,\omega,\psi_{2})$.
\end{prop}
Here we report some properties of the distance $d_{\mathcal{A}}$ and some consequences which will be useful in the sequel.
\begin{prop}
\label{prop:AllNecessary}
The following properties hold:
\begin{itemize}
\item[i)] if $u\in\mathcal{E}^{1}(X,\omega,\psi_{1}), v\in \mathcal{E}^{1}(X,\omega,\psi_{2})$ for $\psi_{1},\psi_{2}\in \overline{\mathcal{A}}$, $\psi_{1}\succcurlyeq \psi_{2}$ then
$$
d_{\mathcal{A}}(u,v)\geq d\big(P_{\omega}[\psi_{2}](u), v\big)
$$
\emph{(Proposition $4.14$, \cite{Tru19})};
\item[ii)] if  $\{\psi_{k}\}_{k\in\mathbbm{N}}\subset \mathcal{M}^{+},\psi\in \mathcal{M}$ with $\psi_{k}\searrow \psi$ (resp. $\psi_{k}\nearrow \psi$ a.e.), $u_{k}\searrow u$, $v_{k}\searrow v$ (resp. $u_{k}\nearrow u$ a.e., $v_{k}\nearrow v$ a.e.) for $u_{k},v_{k}\in \mathcal{E}^{1}(X,\omega,\psi_{k})$, $u,v\in\mathcal{E}^{1}(X,\omega,\psi)$ and $|u_{k}-v_{k}|$ is uniformly bounded, then
$$
d(u_{k},v_{k})\to d(u,v)
$$
\emph{(Lemma $4.6$, \cite{Tru19})};
\item[iii)] if $\{\psi_{k}\}_{k\in\mathbbm{N}}\subset \mathcal{M}^{+},\psi\in \mathcal{M}$ such that $\psi_{k}\to\psi$ monotonically a.e., then for any $\psi'\in \mathcal{M}$ such that $\psi'\succcurlyeq \psi_{k}$ for any $k\gg 1$ big enough, and for any strongly compact set $K\subset \big(\mathcal{E}^{1}(X,\omega,\psi'),d\big)$,
$$
d\big(P_{\omega}[\psi_{k}](\varphi_{1}),P_{\omega}[\psi_{k}](\varphi_{2})\big)\to d\big(P_{\omega}[\psi](\varphi_{1}),P_{\omega}[\psi](\varphi_{2})\big)
$$
uniformly on $K\times K$, i.e. varying $(\varphi_{1},\varphi_{2})\in K\times K$. In particular if $\psi_{k},\psi\in \overline{\mathcal{A}}$ then
\begin{gather*}
d_{\mathcal{A}}\big(P_{\omega}[\psi](u),P_{\omega}[\psi_{k}](u)\big)\to 0\\
d\big(P_{\omega}[\psi_{k}](u),P_{\omega}[\psi_{k}](v)\big)\to d\big(P_{\omega}[\psi](u),P_{\omega}[\psi](v)\big)
\end{gather*}
monotonically for any $(u,v)\in\mathcal{E}^{1}(X,\omega,\psi')\times \mathcal{E}^{1}(X,\omega,\psi')$ \emph{(Proposition $4.5$, \cite{Tru19})};
\item[iv)] $d_{\mathcal{A}}(u_{1},u_{2})\geq |V_{\psi_{1}}-V_{\psi_{2}}|$ if $u_{1}\in\mathcal{E}^{1}(X,\omega,\psi_{1})$, $u_{2}\in\mathcal{E}^{1}(X,\omega,\psi_{2})$ and the equality holds if $u_{1}=\psi_{1}$, $u_{2}=\psi_{2}$ \emph{(by definition of $d_{\mathcal{A}}$, see section \S $4.2$ in \cite{Tru19})}.
\end{itemize}
\end{prop}
The following Lemma is a special case of Theorem $2.2$ in \cite{X19} (see also Lemma $4.1.$ in \cite{DDNL17b}).
\begin{lem}[Proposition $2.7$, \cite{Tru19}]
\label{lem:KeyConv}
Let $\{\psi_{k}\}_{k\in\mathbbm{N}}\subset \mathcal{M}^{+},\psi\in\mathcal{M}$ such that $\psi_{k}\to \psi$ monotonically almost everywhere. Let also $u_{k},v_{k}\in\mathcal{E}^{1}(X,\omega,\psi_{k})$ converging in capacity respectively to $u,v\in\mathcal{E}^{1}(X,\omega,\psi)$. Then for any $j=0,\dots,n$
$$
MA_{\omega}(u_{k}^{j},v_{k}^{n-j})\to MA_{\omega}(u^{j},v^{n-j})
$$
weakly. Moreover if $|u_{k}-v_{k}|$ is uniformly bounded, then for any $j=0,\dots,n$
$$
(u_{k}-v_{k})MA_{\omega}(u_{k}^{j},v_{k}^{n-j})\to (u-v)MA_{\omega}(u^{j},v^{n-j})
$$
weakly.
\end{lem}
It is well-known that the set of Kähler potentials $\mathcal{H}_{\omega}:=\{\varphi\in PSH(X,\omega)\cap C^{\infty}(X)\, : \, \omega+dd^{c}\varphi>0\}$ is dense into $\big(\mathcal{E}^{1}(X,\omega),d\big)$. The same holds for $P_{\omega}[\psi](\mathcal{H}_{\omega})$ into $\big(\mathcal{E}^{1}(X,\omega,\psi),d\big)$.
\begin{lem}[Lemma $4.8$, \cite{Tru19}]
\label{lem:Density}
The set $\mathcal{P}_{\mathcal{H}_{\omega}}(X,\omega,\psi):=P_{\omega}[\psi](\mathcal{H})\subset \mathcal{P}(X,\omega,\psi)$ is dense in $\big(\mathcal{E}^{1}(X,\omega,\psi),d\big)$.
\end{lem}
The following Lemma shows that, for $u\in PSH(X,\omega)$ fixed, the map $\mathcal{M}^{+}\ni\psi\to P_{\omega}[\psi](u)$ is weakly continuous over any totally ordered set of model type envelopes that are more singular than $u$.
\begin{lem}
\label{lem:Referee}
Let $u\in PSH(X,\omega)$, and let $\{\psi_{k}\}_{k\in\mathbbm{N}}\subset \mathcal{M}^{+}$ be a totally ordered sequence of model type envelopes converging to $\psi\in\mathcal{M}$. Assume also that $\psi_{k} \preccurlyeq u$ for any $k\gg 1$ big enough. Then $P_{\omega}[\psi_{k}](u)\to P_{\omega}[\psi](u)$ weakly.
\end{lem}
\begin{proof}
As $\{\psi_{k}\}_{k\in\mathbbm{N}}$ is totally ordered, without loss of generality we may assume that $\psi_{k}\to \psi$ monotonically almost everywhere. Set $\tilde{u}:=\lim_{k\to \infty}P_{\omega}[\psi_{k}](u)$, and we want to prove that $\tilde{u}=P_{\omega}[\psi](u)$.\newline
Suppose $\psi_{k}\searrow \psi$. It is immediate to check that $P_{\omega}[\psi_{k}](u)\leq P_{\omega}[\psi_{k}](\sup_{X}u)=\psi_{k}+\sup_{X}u$, which implies $\tilde{u}\leq \psi+\sup_{X}u$ letting $k\to +\infty$. Thus, $\tilde{u}\leq P_{\omega}[\psi](u)$ as the inequality $\tilde{u}\leq u$ is trivial. Moreover, since $\psi\leq \psi_{k}$ we also have $P_{\omega}[\psi](u)\leq P_{\omega}[\psi_{k}](u)$, which clearly yields $P_{\omega}[\psi](u)\leq \tilde{u}$ and concludes this part.\newline
Suppose $\psi_{k}\nearrow \psi$. Then the inequality $\tilde{u}\leq P_{\omega}[\psi](u)$ is immediate. Next, combining Theorem \ref{thm:3.8} and Proposition \ref{prop:PropProie} we have
\begin{multline*}
MA_{\omega}\big(P_{\omega}[\psi_{k}](u)\big)=MA_{\omega}\Big(P_{\omega}[\psi_{k}]\big(P_{\omega}[\psi](u)\big)\Big)\leq \mathbbm{1}_{\{P_{\omega}[\psi_{k}](u)=P_{\omega}[\psi](u)\}}MA_{\omega}\big(P_{\omega}[\psi](u)\big)\leq\\
\leq\mathbbm{1}_{\{\tilde{u}=P_{\omega}[\psi](u)\}}MA_{\omega}\big(P_{\omega}[\psi](u)\big)
\end{multline*}
where the last inequality follows from $P_{\omega}[\psi_{k}](u)\leq \tilde{u}\leq P_{\omega}[\psi](u)$. Thus, as $MA_{\omega}\big(P_{\omega}[\psi_{k}](u)\big)\to MA_{\omega}(\tilde{u})$ weakly by Theorem $2.3$ in \cite{DDNL17b}, we deduce that $\tilde{u}\in \mathcal{E}(X,\omega,\psi)$ and that $MA_{\omega}(\tilde{u})\leq\mathbbm{1}_{\{\tilde{u}=P_{\omega}[\psi](u)\}}MA_{\omega}\big(P_{\omega}[\psi](u)\big)$. Moreover we also have $P_{\omega}[\psi](u)\in \mathcal{E}(X,\omega,\psi)$. Indeed $P_{\omega}[\psi](u)\leq P_{\omega}[\psi](\sup_{X}u)=\psi+\sup_{X}$, i.e. $P_{\omega}[\psi](u)\preccurlyeq \psi$, while $P_{\omega}[\psi](u)\geq P_{\omega}[\psi](\psi_{k}-C_{k})=\psi_{k}-C_{k}$ for non-negative constants $C_{k}$ and for any $k\gg 1$ big enough as $u,\psi$ are less singular than $\psi_{k}$. Thus $P_{\omega}[\psi](u)\succcurlyeq \psi_{k}$ for any $k$, which yields $\int_{X}MA_{\omega}\big(P_{\omega}[\psi](u)\big)\geq V_{\psi}>0$ and gives $P_{\omega}[\psi](u)\in\mathcal{E}(X,\omega,\psi)$. Hence
$$
0\leq \int_{X}\big(P_{\omega}[\psi](u)-\tilde{u}\big)MA_{\omega}(\tilde{u})\leq \int_{\{\tilde{u}=P_{\omega}[\psi](u)\}}\big(P_{\omega}[\psi](u)-\tilde{u}\big)MA_{\omega}\big(P_{\omega}[\psi](u)\big)=0,
$$
which by the domination principle of Proposition $3.11$ in \cite{DDNL17b} implies $\tilde{u}\geq P_{\omega}[\psi](u)$ and concludes the proof.
\end{proof}
\section{Tools.}
\label{sec:Tools}
In this section we collect some uniform estimates on $\mathcal{E}^{1}(X,\omega,\psi)$ for $\psi\in\mathcal{M}^{+}$, we recall the $\psi$-relative capacity and we will prove the upper semicontinuity of $E_{\cdot}(\cdot)$ on $X_{\mathcal{A}}$.
\subsection{Uniform estimates.}
Let $\psi\in\mathcal{M}^{+}$.\\
We first define in the $\psi$-relative setting the analogous of some well-known functionals of the variational approach (see \cite{BBGZ09} and reference therein).\\

We introduce respectively the \emph{$\psi$-relative $I$-functional} and the \emph{$\psi$-realtive $J$-functional} (see also \cite{Aub84}) $ I_{\psi}, J_{\psi}:\mathcal{E}^{1}(X,\omega,\psi)\times \mathcal{E}^{1}(X,\omega,\psi)\to \mathbbm{R} $ where $\psi\in\mathcal{M}^{+}$ as
\begin{gather*}
I_{\psi}(u,v):=\int_{X}(u-v)\big(MA_{\omega}(v)-MA_{\omega}(u)\big),\\
J_{\psi}(u,v):=J^{\psi}_{u}(v):=E_{\psi}(u)-E_{\psi}(v)+\int_{X}(v-u)MA_{\omega}(u).
\end{gather*}
They assume non-negative values by Proposition \ref{prop:PropertiesE}, $I_{\psi}$ is clearly symmetric while $J_{\psi}$ is convex again by Proposition \ref{prop:PropertiesE}.
Moreover the $\psi$-relative $I$ and $J$ functionals are related each other by the following result.
\begin{lem}
\label{lem:Related}
Let $u,v\in\mathcal{E}^{1}(X,\omega,\psi)$. Then
\begin{itemize}
\item[(i)] $\frac{1}{n+1}I_{\psi}(u,v)\leq J^{\psi}_{u}(v)\leq \frac{n}{n+1}I_{\psi}(u,v)$;
\item[(ii)] $\frac{1}{n} J_{u}^{\psi}(v)\leq J^{\psi}_{v}(u)\leq n J_{u}^{\psi}(v)$.
\end{itemize}
In particular
$$
d(\psi,u)\leq n J^{\psi}_{u}(\psi)+\big(||\psi||_{L^{1}}+||u||_{L^{1}}\big)
$$
for any $u\in\mathcal{E}^{1}(X,\omega,\psi)$ such that $u\leq \psi$.
\end{lem}
\begin{proof}
By Proposition \ref{prop:PropertiesE} it follows that
\begin{multline*}
n\int_{X}(u-v)MA_{\omega}(u) +\int_{X}(u-v)MA_{\omega}(v)\leq\\
\leq (n+1) \big(E_{\psi}(u)-E_{\psi}(v)\big)\leq \int_{X}(u-v)MA_{\omega}(u)+n\int_{X}(u-v)MA_{\omega}(v)
\end{multline*}
for any $u,v\in\mathcal{E}^{1}(X,\omega,\psi)$, which yields $(i)$ and $(ii)$.\\
Next considering $v=\psi$ and assuming $u\leq \psi$ from the second inequality in $(ii)$ we obtain
$$
d(u,\psi)=-E_{\psi}(u)\leq nJ^{\psi}_{u}(\psi)+\int_{X}(\psi-u)MA_{\omega}(\psi),
$$
which implies the assertion since $MA_{\omega}(\psi)\leq MA_{\omega}(0)$ by Theorem \ref{thm:3.8}. 
\end{proof}
We can now proceed showing the uniform estimates, adapting some results in \cite{BBGZ09}.
\begin{lem}[Lemma $3.7$, \cite{Tru19}]
\label{lem:UniEst}
Let $\psi\in\mathcal{M}^{+}$. Then there exists positive constants $A>1$, $B>0$ depending only on $n,\omega$ such that for any $u\in\mathcal{E}^{1}(X,\omega,\psi)$
$$
-d(\psi,u)\leq V_{\psi}\sup_{X}(u-\psi)=V_{\psi}\sup_{X}u\leq Ad(\psi,u)+B
$$
\end{lem}
\begin{rem}
\label{rem:Usef}
\emph{As a consequence of Lemma \ref{lem:UniEst} if $d(\psi,u)\leq C$ then $\sup_{X}u\leq (AC+B)/V_{\psi}$ while $-E_{\psi}(u)=d(\psi+(AC+B)/V_{\psi},u)-(AC+B)\leq d(\psi,u)\leq C$, i.e. $u\in \mathcal{E}^{1}_{D}(X,\omega,\psi)$ where $D:=\max\big(C,(AC+B)/V_{\psi}\big)$. Conversely, it is easy to check that $d(u,\psi)\leq C(2V_{\psi}+1) $ for any $u\in\mathcal{E}^{1}_{C}(X,\omega,\psi)$ using the definitions and the triangle inequality}.
\end{rem}
\begin{prop}
\label{prop:ForL1}
Let $C\in\mathbbm{R}_{>0}$. Then there exists a continuous increasing function $f_{C}:\mathbbm{R}_{\geq 0}\to \mathbbm{R}_{\geq 0}$ depending only on $C,\omega,n$ with $f_{C}(0)=0$ such that
\begin{gather}
\label{eqn:Second2}
\Big|\int_{X}(u-v)\big(MA_{\omega}(\varphi_{1})-MA_{\omega}(\varphi_{2})\big)\Big|\leq  f_{C}\big(d(u,v)\big)
\end{gather}
for any $u,v,\varphi_{1},\varphi_{2}\in\mathcal{E}^{1}(X,\omega,\psi)$ with $d(u,\psi),d(v,\psi),d(\varphi_{1},\psi),d(\varphi_{2},\psi)\leq C$.
\end{prop}
\begin{proof}
As said in Remark \ref{rem:Usef} if $w\in\mathcal{E}^{1}(X,\omega,\psi)$ with $d(\psi,w)\leq C$ then $\tilde{w}:=w-(AC+B)/V_{\psi}$ satisfies $\sup_{X}\tilde{w}\leq 0$ and
$$
-E_{\psi}(\tilde{w})=d(\psi,\tilde{w})\leq d(\psi,w)+d(w,\tilde{w})\leq C+AC+B=:D.
$$
Therefore setting $\tilde{u}:=u-(AC+B)/V_{\psi},\tilde{v}:=v-(AC+B)/V_{\psi}$ we can proceed exactly as in Lemma $5.8$ in \cite{BBGZ09} using the integration by parts formula in \cite{X19} (see also Theorem $1.14$ in \cite{BEGZ10}) to get
\begin{equation}
\label{eqn:First1}
\Big|\int_{X}(\tilde{u}-\tilde{v})\big(MA_{\omega}(\varphi_{1})-MA_{\omega}(\varphi_{2})\big)\Big|\leq I_{\psi}(\tilde{u},\tilde{v})+h_{D}\big(I_{\psi}(\tilde{u},\tilde{v})\big)
\end{equation}
where $h_{D}:\mathbbm{R}_{\geq 0}\to \mathbbm{R}_{\geq 0}$ is an increasing continuous function depending only on $D$ such that $h_{D}(0)=0$. Furthermore, by definition
$$
d\big(\psi,P_{\omega}(\tilde{u},\tilde{v})\big)\leq d(\psi,\tilde{u})+d\big(\tilde{u},P_{\omega}(\tilde{u},\tilde{v})\big)\leq d(\psi,\tilde{u})+d(\tilde{u},\tilde{v})\leq 3D,
$$
so, by the triangle inequality and (\ref{eqn:First1}), we have
\begin{multline}
\label{eqn:ForUni}
\Big|\int_{X}(u-v)\big(MA_{\omega}(\varphi_{1})-MA_{\omega}(\varphi_{2})\big)\Big|\leq\\
\leq  I_{\psi}\big(\tilde{u},P_{\omega}(\tilde{u},\tilde{v})\big)+I_{\psi}\big(\tilde{v},P_{\omega}(\tilde{u},\tilde{v})\big)+h_{3D}\big(I_{\psi}(\tilde{u},P_{\omega}(\tilde{u},\tilde{v}))\big)+h_{3D}\big(I_{\psi}(\tilde{v},P_{\omega}(\tilde{u},\tilde{v}))\big).
\end{multline}
On the other hand, if $w_{1},w_{2}\in\mathcal{E}^{1}(X,\omega,\psi)$ with $w_{1}\geq w_{2}$ then by Proposition \ref{prop:PropertiesE}
\begin{equation*}
I_{\psi}(w_{1},w_{2})\leq \int_{X}(w_{1}-w_{2})MA_{\omega}(w_{2})\leq (n+1) d(w_{1},w_{2}).
\end{equation*}
Hence from (\ref{eqn:ForUni}) it is sufficient to set $f_{C}(x):=(n+1)x+2h_{3D}\big((n+1)x\big)$ to conclude the proof since clearly $d(\tilde{u},\tilde{v})=d(u,v)$.
\end{proof}
\begin{cor}
\label{cor:ForL1}
Let $\psi\in\mathcal{M}^{+}$ and let $C\in\mathbbm{R}_{>0}$. Then there exists a continuous increasing functions $f_{C}:\mathbbm{R}_{\geq 0}\to \mathbbm{R}_{\geq 0}$ depending only on $C, \omega,n$ with $f_{C}(0)=0$ such that
\begin{gather*}
\int_{X}|u-v|MA_{\omega}(\varphi)\leq f_{C}\big(d(u,v)\big)
\end{gather*}
for any $u,v,\varphi\in\mathcal{E}^{1}(X,\omega,\psi)$ with $d(\psi,u),d(\psi,v),d(\psi,\varphi)\leq C$.
\end{cor}
\begin{proof}
Since $d\big(\psi,P_{\omega}(u,v)\big)\leq 3C$, letting $g_{3C}:\mathbbm{R}_{\geq 0}\to \mathbbm{R}_{\geq 0}$ be the map (\ref{eqn:Second2}) of Proposition \ref{prop:ForL1}, it follows that
\begin{multline*}
\int_{X}\big(u-P_{\omega}(u,v)\big)MA_{\omega}(\varphi)\leq \int_{X}\big(u-P_{\omega}(u,v)\big)MA_{\omega}\big(P_{\omega}(u,v)\big)+g_{3C}\Big(d\big(u,P_{\omega}(u,v)\big)\Big)\leq \\
\leq(n+1)d\big(u,P_{\omega}(u,v)\big)+g_{3C}\Big(d(u,v))\Big),
\end{multline*}
where in the last inequality we used Proposition \ref{prop:PropertiesE}. Hence by the triangle inequality we get
\begin{multline*}
\int_{X}|u-v|MA_{\omega}(\varphi)\leq (n+1)d\big(u,P_{\omega}(u,v)\big)+(n+1)d\big(v,P_{\omega}(u,v)\big)+2g_{3C}\big(d(u,v)\big)=\\
=(n+1)d(u,v)+2g_{3C}\big(d(u,v)\big).
\end{multline*}
Defining $f_{C}(x):=(n+1)x+2g_{3C}(x)$ concludes the proof. 
\end{proof}
As first important consequence we obtain that the strong convergence in $\mathcal{E}^{1}(X,\omega,\psi)$ implies the weak convergence.
\begin{prop}
\label{prop:L1}
Let $\psi\in\mathcal{M}^{+}$ and let $C\in\mathbbm{R}_{>0}$. Then there exists a continuous increasing function $f_{C,\psi}:\mathbbm{R}_{\geq 0}\to \mathbbm{R}_{\geq 0}$ depending on $C, \omega,n, \psi $ with $f_{C,\psi}(0)=0$ such that
$$
||u-v||_{L^{1}}\leq f_{C,\psi}\big(d(u,v)\big)
$$
for any $u,v\in\mathcal{E}^{1}(X,\omega,\psi)$ with $d(\psi,u),d(\psi,v)\leq C$. In particular $u_{k}\to u$ weakly if $u_{k}\to u$ strongly.
\end{prop}
\begin{proof}
Theorem $A$ in \cite{DDNL18b} (see also Theorem $1.4$ in \cite{DDNL17b}) implies that there exists $\phi\in\mathcal{E}^{1}(X,\omega,\psi)$ with $\sup_{X}\phi=0$ such that
$$
MA_{\omega}(\phi)=c MA_{\omega}(0)
$$
where $c:=V_{\psi}/V_{0}>0$. Therefore it follows that
$$
||u-v||_{L^{1}}\leq \frac{1}{c}g_{\hat{C}}\big(d(u,v)\big)
$$
where $\hat{C}:=\max\big(d(\psi,\phi),C\big)$ and $g_{\hat{C}}$ is the continuous increasing function with $g_{\hat{C}}(0)=0$ given by Corollary \ref{cor:ForL1}. Setting $f_{C,\psi}:=\frac{1}{c}g_{\hat{C}}$ concludes the proof.
\end{proof}
Finally we also get the following useful estimate.
\begin{prop}
\label{prop:LastLast}
Let $\psi\in\mathcal{M}^{+}$ and let $C\in\mathbbm{R}_{>0}$. Then there exists a constant $\tilde{C}$ depending only on $C,\omega,n$ such that
\begin{equation}
\label{eqn:LastLast}
\Big|\int_{X}(u-v)\big(MA_{\omega}(\varphi_{1})-MA_{\omega}(\varphi_{2})\big)\Big|\leq \tilde{C} \,I_{\psi}(\varphi_{1},\varphi_{2})^{\frac{1}{2}}
\end{equation}
for any $u,v,\varphi_{1},\varphi_{2}\in\mathcal{E}^{1}(X,\omega,\psi)$ with $d(u,\psi),d(v,\psi),d(\varphi_{1},\psi),d(\varphi_{2},\psi)\leq C$.
\end{prop}
\begin{proof}
As seen during the proof of Proposition \ref{prop:ForL1} and with the same notations, the function $\tilde{u}:=u-(AC+B)/V_{\psi}$ satisfy $\sup_{X}u\leq 0$ (by Lemma \ref{lem:UniEst}) and $-E_{\psi}(u)\leq C+AC+B=:D$ (and similarly for $v,\varphi_{1},\varphi_{2}$). Therefore by integration by parts and using Lemma \ref{lem:UniBounded} below, it follows exactly as in Lemma $3.13$ in \cite{BBGZ09} that there exists a constant $\tilde{C}$ depending only on $D,n$ such that 
$$
\Big|\int_{X}(\tilde{u}-\tilde{v})\big(MA_{\omega}(\tilde{\varphi}_{1})-MA_{\omega}(\tilde{\varphi}_{2})\big)\Big|\leq \tilde{C} \,I_{\psi}(\tilde{\varphi}_{1},\tilde{\varphi}_{2})^{\frac{1}{2}},
$$
which clearly implies (\ref{eqn:LastLast}).
\end{proof}
\begin{lem}
\label{lem:UniBounded}
Let $C\in\mathbbm{R}_{>0}$. Then there exists a constant $\tilde{C}$ depending only on $C,\omega,n$ such that
$$
\int_{X}|u_{0}-\psi|(\omega+dd^{c}u_{1})\wedge \cdots \wedge (\omega+dd^{c}u_{n})\leq \tilde{C}
$$
for any $u_{0},\cdots,u_{n}\in\mathcal{E}^{1}(X,\omega,\psi)$ with $d(u_{j},\psi)\leq C$ for any $j=0,\dots,n$.
\end{lem}
\begin{proof}
As in Proposition \ref{prop:ForL1} and with the same notations $v_{j}:=u_{j}-(AC+B)/V_{\psi}$ satisfies $\sup_{X}v_{j}\leq 0$, and setting $v:=\frac{1}{n+1}(v_{0}+\cdots+v_{n})$ we obtain $\psi-u_{0}\leq (n+1)(\psi-v)$. Thus by Proposition \ref{prop:PropertiesE} it follows that
\begin{gather*}
\int_{X}(\psi-v_{0})MA_{\omega}(v)\leq (n+1)\int_{X}(\psi-v)MA_{\omega}(v)\leq (n+1)^{2}|E_{\psi}(v)|\leq\\
\leq (n+1)\sum_{j=0}^{n}|E_{\psi}(v_{j})|\leq (n+1)\sum_{j=0}^{n}\big(d(\psi,u_{j})+D\big)\leq (n+1)^{2}(C+D)
\end{gather*}
where $D:=AC+B$. On the other hand $MA_{\omega}(v)\geq E (\omega+dd^{c}u_{1})\wedge\cdots (\omega+dd^{c}u_{n})$ where the constant $E$ depends only on $n$. Finally we get
\begin{multline*}
\int_{X}|u_{0}-\psi|(\omega+dd^{c}u_{1})\wedge \cdots\wedge (\omega+dd^{c}u_{n})\leq D+\frac{1}{E}\int_{X}(\psi-v_{0})MA_{\omega}(v)\leq D+\frac{(n+1)^{2}(C+D)}{E},
\end{multline*}
which concludes the proof.
\end{proof}
\subsection{$\psi$-relative Monge-Ampère capacity.}
\begin{defn}[Section \S 4.1, \cite{DDNL17b}; Definition $3.1$, \cite{DDNL18b}]
Let $B\subset X$ be a Borel set, and let $\psi\in \mathcal{M}^{+}$. Then its $\psi$-\emph{relative Monge-Ampère capacity} is defined as
$$
\mathrm{Cap}_{\psi}(B):=\sup\Big\{\int_{B}MA_{\omega}(u)\, :\, u\in PSH(X,\omega),\,  \psi-1\leq u\leq \psi\Big\}.
$$
\end{defn}

In the absolute setting the Monge-Ampère capacity is very useful to study the existence and the regularity of solutions of degenerate complex Monge-Ampère equation (\cite{Kol98}), and analog holds in the relative setting (\cite{DDNL17b}, \cite{DDNL18b}). We refer to these articles just cited to many properties of the Monge-Ampère capacity.\\
For any fixed constant $A$, $\mathcal{C}_{A,\psi}$ denotes the set of all probability measures $\mu$ on $X$ such that
$$
\mu(B)\leq A\mbox{Cap}_{\psi}(B)
$$
for any Borel set $B\subset X$ (Section \S 4.3, \cite{DDNL17b}).
\begin{prop}
\label{prop:CapMS}
Let $u\in\mathcal{E}^{1}(X,\omega,\psi)$ with $\psi$-relative minimal singularities. Then $MA_{\omega}(u)/V_{\psi}\in \mathcal{C}_{A,\psi}$ for a constant $A>0$.
\end{prop}
\begin{proof}
Let $j\in\mathbbm{R}$ such that $u\geq \psi-j$ and assume without loss of generality that $u\leq \psi$ and that $j\geq 1$. Then the function $v:=j^{-1}u+(1-j^{-1})\psi$ is a candidate in the definition of $\mathrm{Cap}_{\psi}$, which implies that $ MA_{\omega}(v)\leq \mathrm{Cap}_{\psi}$. Hence, since $MA_{\omega}(u)\leq j^{n}MA(v)$ we get that $MA_{\omega}(u)\in \mathcal{C}_{A,\psi}$ for $A=j^{n}$ and the result follows.
\end{proof}
We also need to quote the following result.
\begin{lem}[Lemma $4.18$, \cite{DDNL17b}]
\label{lem:4.18}
If $\mu\in \mathcal{C}_{A,\psi}$ then there is a constant $B>0$ depending only on $A,n$ such that
$$
\int_{X}(u-\psi)^{2}\mu\leq B\big(|E_{\psi}(u)|+1\big)
$$
for any $u\in PSH(X,\omega,\psi)$ such that $\sup_{X}u=0$.
\end{lem}
Similarly to the case $\psi=0$ (see \cite{GZ17}), we say that a sequence $u_{k}\in PSH(X,\omega) $ converges to $u\in PSH(X,\omega)$ in $\psi$-relative capacity for $\psi\in\mathcal{M}$ if
$$
\mathrm{Cap}_{\psi}\big(\{|u_{k}-u|\geq \delta\}\big)\to 0
$$
as $k\to \infty$ for any $\delta>0$.\\
By Theorem $10.37$ in \cite{GZ17} (see also Theorem $5.7$ in \cite{BBGZ09}) the convergence in $\big(\mathcal{E}^{1}(X,\omega),d\big)$  implies the convergence in capacity. The analogous holds for $\psi\in\mathcal{M}^{+}$, i.e. that the strong convergence in $\mathcal{E}^{1}(X,\omega,\psi)$ implies the convergence in $\psi$-relative capacity. Indeed in Proposition \ref{prop:CapacityPsi} we will prove the the strong convergence implies the convergence in $\psi'$-relative capacity for any $\psi'\in\mathcal{M}^{+}$.
\subsection{(Weak) Upper Semicontinuity of $u\to E_{P_{\omega}[u]}(u)$ over $X_{\mathcal{A}}$.}
One of the main feature of $E_{\psi}$ for $\psi\in \mathcal{M}$ is its upper semicontinuity with respect to the weak topology. Here we prove the analogous for $E_{\cdot}(\cdot)$ over $X_{\mathcal{A}}$.
\begin{lem}
\label{lem:HomeoV}
The map $MA_{\omega}:\overline{\mathcal{A}}\to MA_{\omega}(\overline{\mathcal{A}})\subset \{\mu\,\, \mbox{positive measure on}\,\, X\}$ is a homeomorphism considering the weak topologies. In particular $\overline{\mathcal{A}}$ is homeomorphic to a closed set contained in $[0,\int_{X}MA_{\omega}(0)]$ through the map $\psi\to V_{\psi}$.
\end{lem}
\begin{proof}
The map is well-defined and continuous by Lemma $2.6$ in \cite{Tru19}. Moreover the injectivity follows from the fact that $V_{\psi_{1}}=V_{\psi_{2}}$ for $\psi_{1},\psi_{2}\in\overline{\mathcal{A}}$ implies $\psi_{1}=\psi_{2}$ using Theorem \ref{thm:ClassE} and the fact that $\mathcal{A}\subset \mathcal{M}^{+}$. \\
Finally to conclude the proof it is enough to prove that $\psi_{k}\to \psi$ weakly assuming $V_{\psi_{k}}\to V_{\psi}$ and it is clearly sufficient to show that any subsequence of $\{\psi_{k}\}_{k\in\mathbbm{N}}$ admits a subsequence weakly convergent to $\psi$. Moreover since $\overline{\mathcal{A}}$ is totally ordered and $\succcurlyeq$ coincides with $\geq $ on $\mathcal{M}$, we may assume $\{\psi_{k}\}_{k\in\mathbbm{N}}$ monotonic sequence. Then, up to considering a further subsequence, $\psi_{k}$ converges almost everywhere to an element $\psi'\in\overline{\mathcal{A}}$ by compactness, and Lemma \ref{lem:KeyConv} implies that $V_{\psi'}=V_{\psi}$, i.e $\psi=\psi'$.
\end{proof}
In the case $\mathcal{A}:=\{\psi_{k}\}_{k\in\mathbbm{N}}\subset \mathcal{M}^{+}$, we say that $u_{k}\in\mathcal{E}^{1}(X,\omega,\psi_{k})$ converges weakly to $P_{\psi_{\min}}$ where $\psi_{\min}\in\mathcal{M}\setminus\mathcal{M}^{+}$ if $|\sup_{X}u_{k}|\leq C$ for any $k\in\mathbbm{N}$ and any weak accumulation point $u$ of $\{u_{k}\}_{k\in\mathbbm{N}}$ satisfies $u\preccurlyeq \psi_{\min}$. This definition is the most natural since $PSH(X,\omega,\psi)=\mathcal{E}^{1}(X,\omega,\psi_{\min})$.
\begin{lem}
\label{lem:Contra}
Let $\{u_{k}\}_{k\in\mathbbm{N}}\subset X_{\mathcal{A}}$ be a sequence converging weakly to $u\in X_{\mathcal{A}}$. If $E_{P_{\omega}[u_{k}]}(u_{k})\geq C$ uniformly, then $P_{\omega}[u_{k}]\to P_{\omega}[u]$ weakly.
\end{lem}
\begin{proof}
By Lemma \ref{lem:HomeoV} the convergence requested is equivalent to $V_{\psi_{k}}\to V_{\psi}$, where we set $\psi_{k}:=P_{\omega}[u_{k}], \psi:=P_{\omega}[u]$.\\
Moreover by a simple contradiction argument it is enough to show that any subsequence $\{\psi_{k_{h}}\}_{h\in\mathbbm{N}}$ admits a subsequence $\{\psi_{k_{h_{j}}}\}_{j\in\mathbbm{N}}$ such that $V_{\psi_{k_{h_{j}}}}\to V_{\psi}$. Thus up to considering a subsequence, by abuse of notations and by the lower semicontinuity $\liminf_{k\to \infty}V_{\psi_{k}}\geq V_{\psi}$ of Theorem $2.3.$ in \cite{DDNL17b}, we may suppose by contradiction that $\psi_{k}\searrow \psi'$ for $\psi'\in \mathcal{M}$ such that $V_{\psi'}>V_{\psi}$. In particular $V_{\psi'}>0$ and $\psi'\succcurlyeq  \psi$. Then by Proposition \ref{prop:PropProie} and Remark \ref{rem:Usef} the sequence $\{P_{\omega}[\psi'](u_{k})\}_{k\in\mathbbm{N}}$ is bounded in $\big(\mathcal{E}^{1}(X,\omega,\psi'),d\big)$ and it belongs to $\mathcal{E}^{1}_{C'}(X,\omega,\psi')$ for some $C'\in\mathbbm{R}$. Therefore, up to considering a subsequence, we have that $\{u_{k}\}_{k\in\mathbbm{N}}$ converges weakly to an element $v\in \mathcal{E}^{1}(X,\omega,\psi)$ (which is the element $u$ itself when $u\neq P_{\psi_{\min}}$) while the sequence $P_{\omega}[\psi'](u_{k})$ converges weakly to an element $w \in\mathcal{E}^{1}(X,\omega,\psi')$. Thus the contradiction follows from $w\leq v$ since $\psi'\succcurlyeq \psi$, $V_{\psi'}>0$ and $\mathcal{E}^{1}(X,\omega,\psi')\cap \mathcal{E}^{1}(X,\omega,\psi)=\emptyset$.
\end{proof}
\begin{prop}
\label{prop:USC}
Let $\{u_{k}\}_{k\in\mathbbm{N}}\subset X_{\mathcal{A}}$ be a sequence converging weakly to $u\in X_{\mathcal{A}}$. Then
\begin{equation}
\label{eqn:USC}
\limsup_{k\to \infty }E_{P_{\omega}[u_{k}]}(u_{k})\leq E_{P_{\omega}[u]}(u).
\end{equation}
\end{prop}
\begin{proof}
Let $\psi_{k}:=P_{\omega}[u_{k}],\psi:=P_{\omega}[u]\in \overline{\mathcal{A}}$. We may clearly assume $\psi_{k}\neq \psi_{\min}$ for any $k\in\mathbbm{N}$ if $\psi=\psi_{\min}$ and $V_{\psi_{\min}}=0$.\\
Moreover we can also suppose that $E_{\psi_{k}}(u_{k})$ is bounded from below, which implies that $u_{k}\in \mathcal{E}^{1}_{C}(X,\omega,\psi_{k})$ for a uniform constant $C$ and that $\psi_{k}\to \psi$ weakly by Lemma \ref{lem:Contra}. Thus since $E_{\psi_{k}}(u_{k})=E_{\psi_{k}}(u_{k}-C)+CV_{\psi_{k}}$ for any $k\in\mathbbm{N}$, Lemma \ref{lem:HomeoV} implies that we may assume that $\sup_{X}u_{k}\leq 0$. Furthermore since $\mathcal{A}$ is totally ordered, it is enough to show (\ref{eqn:USC}) when $\psi_{k}\to \psi$ a.e. monotonically.\\
If $\psi_{k}\searrow \psi$, setting $v_{k}:=\big(\sup\{u_{j}\, :\, j\geq k\}\big)^{*}\in\mathcal{E}^{1}(X,\omega,\psi_{k})$, we easily have
$$
\limsup_{k\to\infty}E_{\psi_{k}}(u_{k})\leq \limsup_{k\to \infty} E_{\psi_{k}}(v_{k})\leq\limsup_{k\to \infty}E_{\psi}\big(P_{\omega}[\psi](v_{k})\big)
$$
using the monotonicity of $E_{\psi_{k}}$ and Proposition \ref{prop:PropProie}. Hence if $\psi=\psi_{\min}$ and $V_{\psi_{\min}}=0$ then $E_{\psi}\big(P_{\omega}[\psi](v_{k})\big)=0=E_{\psi}(u)$, while otherwise the conclusion follows from Proposition \ref{prop:PropertiesE} since $P_{\omega}[\psi](v_{k})\searrow u$ by construction.\\
If instead $\psi_{k}\nearrow \psi$, fix $\epsilon>0$ and for any $k\in\mathbbm{N}$ let $j_{k}\geq k$ such that
$$
\sup_{j\geq k}E_{\psi_{j}}(u_{j})\leq E_{\psi_{j_{k}}}(u_{j_{k}})+\epsilon.
$$
Thus again by Proposition \ref{prop:PropProie}, $E_{\psi_{j_{k}}}(u_{j_{k}})\leq E_{\psi_{l}}\big(P_{\omega}[\psi_{l}](u_{j_{k}})\big)$ for any $l\leq j_{k}$. Moreover, assuming $E_{\psi_{j_{k}}}(u_{j_{k}})$ bounded from below, $-E_{\psi_{l}}\big(P_{\omega}[\psi_{l}](u_{j_{k}})\big)=d\big(\psi_{l},P_{\omega}[\psi_{l}](u_{j_{k}})\big)$ is uniformly bounded in $l,k$, which implies that $\sup_{X}P_{\omega}[\psi_{l}](u_{j_{k}})$ is uniformly bounded by Remark \ref{rem:Usef} since $V_{\psi_{j_{k}}}\geq a>0$ for $k\gg 0$ big enough. By compactness, up to considering a subsequence, we obtain $P_{\omega}[\psi_{l}](u_{j_{k}})\to v_{l}$ weakly where $v_{l}\in\mathcal{E}^{1}(X,\omega,\psi_{l})$ by the upper semicontinuity of $E_{\psi_{l}}(\cdot)$ on $\mathcal{E}^{1}(X,\omega, \psi_{l})$. Hence
$$
\limsup_{k\to \infty}E_{\psi_{k}}(u_{k})\leq \limsup_{k\to \infty}E_{\psi_{l}}\big(P_{\omega}[\psi_{l}](u_{j_{k}})\big)+\epsilon= E_{\psi_{l}}(v_{l}) +\epsilon
$$
for any $l\in\mathbbm{N}$. Moreover by construction $v_{l}\leq P_{\omega}[\psi_{l}](u)$ since $P_{\omega}[\psi_{l}](u_{j_{k}})\leq u_{j_{k}}$ for any $k$ such that $j_{k}\geq l$ and $u_{j_{k}}\to u$ weakly. Therefore by the monotonicity of $E_{\psi_{l}}(\cdot)$ and by Proposition \ref{prop:AllNecessary}$.(ii)$ we conclude that
$$
\limsup_{k\to \infty}E_{\psi_{k}}(u_{k})\leq \lim_{l\to \infty}E_{\psi_{l}}\big(P_{\omega}[\psi_{l}](u)\big)+\epsilon=E_{\psi}(u)+\epsilon
$$
letting $l\to \infty$.
\end{proof}
As a consequence, defining
$$
X_{\mathcal{A},C}:=\bigsqcup_{\psi\in\overline{\mathcal{A}}}\mathcal{E}^{1}_{C}(X,\omega,\psi),
$$ 
we get the following compactness result.
\begin{prop}
\label{prop:CompactL1}
Let $C,a\in\mathbbm{R}_{>0}$. The set
$$
X_{\mathcal{A},C}^{a}:=X_{\mathcal{A},C}\cap\Big( \bigsqcup_{\psi\in\overline{\mathcal{A}}\, : \, V_{\psi}\geq a}\mathcal{E}^{1}(X,\omega,\psi)\Big)
$$
is compact with respect to the weak topology.
\end{prop}
\begin{proof}
It follows directly from the definition that
$$
X_{\mathcal{A},C}^{a}\subset \Big\{u\in PSH(X,\omega)\, :\, |\sup_{X}u|\leq C' \Big\}
$$
where $C':=\max(C,C/a)$. Therefore by Proposition $8.5$ in \cite{GZ17}, $X_{\mathcal{A},C}^{a}$ is weakly relatively compact. Finally Proposition \ref{prop:USC} and Hartogs' Lemma imply that $X_{\mathcal{A},C}^{a}$ is also closed with respect to the weak topology, concluding the proof.
\end{proof}
\begin{rem}
\label{rem:ImpRem}
\emph{The whole set $X_{\mathcal{A},C}$ may not be weakly compact. Indeed assuming $V_{\psi_{\min}}=0$ and letting $\psi_{k}\in\overline{\mathcal{A}}$ such that $\psi_{k}\searrow \psi_{\min}$, the functions $u_{k}:=\psi_{k}-1/\sqrt{V_{\psi_{k}}}$ belong to $X_{\mathcal{A},V}$ for $V=\int_{X}MA_{\omega}(0)$ since $E_{\psi_{k}}(u_{k})=-\sqrt{V_{\psi_{k}}}$ but $\sup_{X}u_{k}=-1/\sqrt{V_{\psi_{k}}}\to -\infty$.}
\end{rem}
\section{The action of measures on $PSH(X,\omega)$.}
\label{sec:ActionMeasure}
In this section we want to replace the action on $PSH(X,\omega)$ defined in $\cite{BBGZ09}$ given by a probability measure $\mu$ with an action which assume finite values on elements $u\in PSH(X,\omega)$ with $\psi$-relative minimal singularities where $\psi=P_{\omega}[u]$ for almost all $\psi\in\mathcal{M}$. On the other hand for any $\psi\in\mathcal{M}$ we want that there exists many measures $\mu$ whose action over $\{u\in PSH(X,\omega)\, : \, P_{\omega}[u]=\psi\}$ is well-defined. The problem is that $\mu$ varies among \emph{all} probability measures while $\psi$ among \emph{all} model type envelopes. So it may happen that $\mu$ takes mass on non-pluripolar sets and that the unbounded locus of $\psi\in\mathcal{M}$ is very nasty.
\begin{defn}
Let $\mu$ be a probability measure on $X$. Then $\mu$ acts on $PSH(X,\omega)$ through the functional $L_{\mu}:PSH (X,\omega)\to \mathbbm{R}\cup\{-\infty\}$ defined as $L_{\mu}(u)=-\infty$ if $\mu$ charges $\{P_{\omega}[u]=-\infty\}$, as
$$
L_{\mu}(u):=\int_{X}\big(u-P_{\omega}[u]\big)\mu
$$
if $u$ has $P_{\omega}[u]$-relative minimal singularities and $\mu$ does not charge $\{P_{\omega}[u]=-\infty\}$ and as
$$
L_{\mu}(u):=\inf\{L_{\mu}(v)\, : \, v\in PSH(X,\omega)\,\, \mbox{with} \,\, P_{\omega}[u]\mbox{-relative minimal singularities,}\,\, v\geq u\}
$$
otherwise.
\end{defn}
\begin{prop}
\label{prop:PropertiesL}
The following properties hold:
\begin{itemize}
\item[(i)] $L_{\mu}$ is affine, i.e. it satisfies the scaling property $L_{\mu}(u+c)=L_{\mu}(u)+c$ for any $c\in\mathbbm{R}$, $u\in PSH(X,\omega)$;
\item[(ii)] $L_{\mu}$ is non-decreasing on $\{u\in PSH(X,\omega)\, :\, P_{\omega}[u]=\psi\}$ for any $\psi\in\mathcal{M}$;
\item[(iii)] $L_{\mu}(u)=\lim_{j\to \infty}L_{\mu}\big(\max(u,P_{\omega}[u]-j)\big)$ for any $u\in PSH(X,\omega)$;
\item[(iv)] if $\mu$ is non-pluripolar then $L_{\mu}$ is convex;
\item[(v)] if $\mu$ is non-pluripolar and $u_{k}\to u$ and $P_{\omega}[u_{k}]\to P_{\omega}[u]$ weakly as $k\to \infty$ then $L_{\mu}(u)\geq \limsup_{k\to \infty}L_{\mu}(u_{k})$;
\item[(vi)] if $u\in\mathcal{E}^{1}(X,\omega,\psi)$ for $\psi\in\mathcal{M}^{+}$ then $L_{MA_{\omega}(u)/V_{\psi}}$ is finite on $\mathcal{E}^{1}(X,\omega,\psi)$.
\end{itemize}
\end{prop}
\begin{proof}
The first two points follow by definition.\\
For the third point, setting $\psi:=P_{\omega}[u]$, clearly $L_{\mu}(u)\leq \lim_{j\to \infty}L_{\mu}\big(\max(u,\psi-j)\big)$. Conversely, for any $v\geq u$ with $\psi$-relative minimal singularities $v\geq \max(u,\psi-j)$ for $j\gg 0$ big enough, hence by $(ii)$ we get $L_{\mu}(v)\geq \lim_{j\to \infty}L_{\mu}\big(\max(u,\psi-j)\big)$ which implies $(iii)$ by definition.\\
Next, we prove $(iv)$. Let $v=\sum_{l=1}^{m}a_{l}u_{l}$ be a convex combination of elements $u_{l}\in PSH(X,\omega)$, and without loss of generality we may assume $\sup_{X}v,\sup_{X}u_{l}\leq 0$. In particular we have $L_{\mu}(v),L_{\mu}(u_{l})\leq 0$.\\
Suppose $L_{\mu}(v)>-\infty$ (otherwise it is trivial) and let $\psi:=P_{\omega}[v]$, $\psi_{l}:=P_{\omega}[u_{l}]$. Then for any $C\in\mathbbm{R}_{>0}$ it is easy to see that
$$
\sum_{l=1}^{m}a_{l}P_{\omega}(u_{l}+C,0)\leq P_{\omega}(v+C,0)\leq \psi,
$$
which leads to $\sum_{l=1}^{m}a_{l}\psi_{l}\leq \psi$ letting $C\to \infty$. Hence $(iii)$ yields
$$
-\infty<L_{\mu}(v)=\int_{X}(v-\psi)\mu\leq \sum_{l=1}^{n}a_{l}\int_{X}(u_{l}-\psi_{l})\mu=\sum_{l=1}^{n}a_{l}L_{\mu}(u_{l}).
$$
The point $(v)$ easily follows from $\limsup_{k\to \infty}\max\big(u_{k},P_{\omega}[u_{k}]-j\big)\leq \max\big(u,P_{\omega}[u]-j\big)$ and $(iii)$, while the last point is a consequence of Lemma \ref{lem:UniBounded}.
\end{proof}
Next, since for any $t\in [0,1]$ and any $u,v\in\mathcal{E}^{1}(X,\omega,\psi)$
\begin{multline*}
\int_{X}(u-v)MA_{\omega}\big(tu+(1-t)v\big)=\\
=(1-t)^{n}\int_{X}(u-v)MA_{\omega}(v)+\sum_{j=1}^{n}\binom{n}{j}t^{j}(1-t)^{n-j}\int_{X}(u-v)MA_{\omega}(u^{j},v^{n-j})\geq\\
\geq (1-t)^{n}\int_{X}(u-v)MA_{\omega}(v)+\big(1-(1-t)^{n}\big)\int_{X}(u-v)MA_{\omega}(u),
\end{multline*}
we can proceed exactly as in Proposition $3.4$ in \cite{BBGZ09} (see also Lemma $2.11.$ in \cite{GZ07}), replacing $V_{\theta}$ with $\psi$, to get the following result.
\begin{prop}
\label{prop:Imp}
Let $A\subset PSH(X,\omega)$ and let $L:A\to \mathbbm{R}\cup\{-\infty\}$ be a convex and non-decreasing function satisfying the scaling property $L(u+c)=L(u)+c$ for any $c\in\mathbbm{R}$. Then
\begin{itemize}
\item[(i)] if $L$ is finite valued on a weakly compact convex set $K\subset A$, then $L(K)$ is bounded;
\item[(ii)] if $ \mathcal{E}^{1}(X,\omega,\psi)\subset A$ and $L$ is finite valued on $\mathcal{E}^{1}(X,\omega,\psi)$ then $\sup_{\{u\in \mathcal{E}_{C}^{1}(X,\omega,\psi)\, :\, \sup_{X}u\leq 0\}}|L|=O(C^{1/2})$ as $C\to \infty$.
\end{itemize}
\end{prop}
\subsection{When is $L_{\mu}$ continuous?}
The continuity of $L_{\mu}$ is a hard problem. However we can characterize its continuity on some weakly compact sets as the next Theorem shows.
\begin{thm}
\label{thm:3.10}
Let $\mu$ be a non-pluripolar probability measure, and let $K\subset PSH(X,\omega)$ be a compact convex set such that $L_{\mu}$ is finite on $K$, the set $\{P_{\omega}[u]\, :\, u\in K\}\subset \mathcal{M}$ is totally ordered and its closure in $PSH(X,\omega)$ has at most one element in $\mathcal{M}\setminus \mathcal{M}^{+}$. Suppose also that there exists $C\in\mathbbm{R}$ such that $|E_{P_{\omega}[u]}(u)|\leq C$ for any $u\in K$. Then the following properties are equivalent:
\begin{itemize}
\item[(i)] $L_{\mu}$ is continuous on $K$;
\item[(ii)] the map $\tau: K\to L^{1}(\mu)$, $\tau(u):=u-P_{\omega}[u]$ is continuous;
\item[(iii)] the set $\tau(K)\subset L^{1}(\mu)$ is uniformly integrable, i.e.
$$
\int_{t=m}^{\infty}\mu\{u\leq P_{\omega}[u]-t\}\to 0
$$
as $m\to \infty$, uniformly for $u\in K$. 
\end{itemize}
\end{thm}
\begin{proof}
We first observe that if $u_{k}\in K$ converges to $u\in K$ then by Lemma \ref{lem:Contra} $\psi_{k}\to \psi$ where we set $\psi_{k}:=P_{\omega}[u_{k}], \psi:=P_{\omega}[u]$.\\
Then we can proceed exactly as in Theorem $3.10$ in \cite{BBGZ09} to get the equivalence between $(i)$ and $(ii)$, $(ii)\Rightarrow (iii)$ and the fact that the graph of $\tau$ is closed. It is important to underline that $(iii)$ is equivalent to say that $\tau(K)$ is \emph{weakly} relative compact by Dunford-Pettis Theorem, i.e. with respect to the weak topology on $L^{1}(\mu)$ induced by $L^{\infty}(\mu)=L^{1}(\mu)^{*}$.\\
Finally assuming that $(iii)$ holds, it remains to prove $(i)$. So, letting $u_{k},u\in K$ such that $u_{k}\to u$, we have to show that $\int_{X}\tau(u_{k})\mu\to \int_{X}\tau(u)\mu$. Since $\tau(K)\subset L^{1}(\mu)$ is bounded, unless considering a subsequence, we may suppose $\int_{X}\tau(u_{k})\to L\in\mathbbm{R}$. By Fatou's Lemma, 
\begin{equation}
\label{eqn:1Fir}
L=\lim_{k\to \infty}\int_{X}\tau(u_{k})\mu\leq \int_{X}\tau(u)\mu.
\end{equation}
Then for any $k\in\mathbbm{N}$ the closed convex envelope
$$
C_{k}:=\overline{\mbox{Conv}\{\tau(u_{j})\, : \, j\geq k\}},
$$
is weakly closed in $L^{1}(\mu)$ by Hahn-Banach Theorem, which implies that $C_{k}$ is weakly compact since it is contained in $\tau(K)$. Thus since $C_{k}$ is a decreasing sequence of non-empty weakly compact sets, there exists $f\in \bigcap_{k\geq 1}C_{k}$ and there exist elements $v_{k}\in\mbox{Conv}(u_{j}\, :\, j\geq k)$ given as finite convex combination such that $\tau(v_{k})\to f$ in $L^{1}(\mu)$. Moreover by the closed graph property $f=\tau(u)$ since $v_{k}\to u$ as a consequence of $u_{k}\to u$. On the other hand by Proposition \ref{prop:PropertiesL}$.(iv)$ we get
$$
\int_{X}\tau(v_{k})\mu\leq\sum_{l=1}^{m_{k}}a_{l,k}\int_{X}\tau(u_{k_{l}})\mu
$$
if $v_{k}=\sum_{l=1}^{m_{k}}a_{l,k}u_{k_{l}}$.
Hence $L\geq\int_{X}\tau(u)\mu$, which together (\ref{eqn:1Fir}) implies $L=\int_{X}\tau(u)\mu$ and concludes the proof.
\end{proof}
\begin{cor}
\label{cor:Useful}
Let $\psi\in\mathcal{M}^{+}$ and $\mu\in\mathcal{C}_{A,\psi}$. Then $L_{\mu}$ is continuous on $\mathcal{E}^{1}_{C}(X,\omega,\psi)$ for any $C\in\mathbbm{R}_{>0}$. In particular if $\mu=MA_{\omega}(u)/V_{\psi}$ for $u\in\mathcal{E}^{1}(X,\omega,\psi)$ with $\psi$-relative minimal singularities then $L_{\mu}$ is continuous on $\mathcal{E}^{1}_{C}(X,\omega,\psi)$ for any $C\in\mathbbm{R}_{>0}$. 
\end{cor}
\begin{proof}
With the notations of Theorem \ref{thm:3.10}, $\tau\big(\mathcal{E}_{C}^{1}(X,\omega,\psi)\big)$ is bounded in $L^{2}(\mu)$ by Lemma \ref{lem:4.18}. Hence by Holder's inequality $\tau\big(\mathcal{E}^{1}_{C}(X,\omega,\psi)\big)$ is uniformly integrable and Theorem \ref{thm:3.10} yields the continuity of $L_{\mu}$ on $\mathcal{E}^{1}_{C}(X,\omega,\psi)$ for any $C\in\mathbbm{R}_{>0}$.\\
The last assertion follows directly from Proposition \ref{prop:CapMS}.
\end{proof}
The following Lemma will be essential to prove Theorem \ref{thmA}, Theorem \ref{thmB}.
\begin{lem}
\label{lem:Equicontinuity}
Let $\varphi\in\mathcal{H}_{\omega}$ and let $\mathcal{A}\subset \mathcal{M}$ be a totally ordered subset. Set also $v_{\psi}:=P_{\omega}[\psi](\varphi)$ for any $\psi\in\mathcal{A}$. Then the actions $\{V_{\psi}L_{MA_{\omega}(v_{\psi})/V_{\psi}}\}_{\psi\in\mathcal{A}}$ take finite values and they are equicontinuous on any compact set $K\subset PSH(X,\omega)$ such that $\{P_{\omega}[u]\, :\, u\in K\}$ is a totally ordered set whose closure in $PSH(X,\omega)$ has at most one element in $\mathcal{M}\setminus\mathcal{M}^{+}$ and such that $|E_{P_{\omega}[u]}(u)|\leq C$ uniformly for any $u\in K$. If $\psi\in\mathcal{M}\setminus \mathcal{M}^{+}$, for the action $V_{\psi}L_{MA_{\omega}(v_{\psi})/V_{\psi}}$ we mean the null action. In particular if $\psi_{k}\to \psi$ monotonically almost everywhere and $\{u_{k}\}_{k\in\mathbbm{N}}\subset K$ converges weakly to $u\in K$, then
\begin{gather}
\label{eqn:Mab}
\int_{X}\big(u_{k}-P_{\omega}[u_{k}]\big)MA_{\omega}(v_{\psi_{k}})\to \int_{X}\big(u-P_{\omega}[u]\big)MA_{\omega}(v_{\psi}).
\end{gather}
\end{lem}
\begin{proof}
By Theorem \ref{thm:3.8}, $\Big|V_{\psi}L_{MA_{\omega}(v_{\psi})/V_{\psi}}(u)\Big|\leq\int_{X}|u-P_{\omega}[u]|MA_{\omega}(\varphi)$ for any $u\in PSH(X,\omega)$ and any $\psi\in\mathcal{A}$, so the actions in the statement assume finite values. Then the equicontinuity on any weak compact set $K\subset PSH(X,\omega)$ satisfying the assumptions of the Lemma follows from
$$
V_{\psi}\Big|L_{MA_{\omega}(v_{\psi})/V_{\psi}}(w_{1})-L_{MA_{\omega}(v_{\psi})/V_{\psi}}(w_{2})\Big|\leq\int_{X}\big|w_{1}-P_{\omega}[w_{1}]-w_{2}+P_{\omega}[w_{2}]\big|MA_{\omega}(\varphi)
$$
for any $w_{1},w_{2}\in PSH(X,\omega)$ since $MA_{\omega}(\varphi)$ is a volume form on $X$ and $P_{\omega}[w_{k}]\to P_{\omega}[w]$ if $\{w_{k}\}_{k\in\mathbbm{N}}\subset K$ converges to $w\in K$ under our hypothesis by Lemma \ref{lem:Contra}.\\
For the second assertion, if $\psi_{k}\searrow \psi$ (resp. $\psi_{k}\nearrow \psi$ almost everywhere), letting $f_{k},f\in L^{\infty}$ such that $MA_{\omega}(v_{\psi_{k}})=f_{k}MA_{\omega}(\varphi)$ and $MA_{\omega}(v_{\psi})=f MA_{\omega}(\varphi)$ (Theorem \ref{thm:3.8}), we have $0\leq f_{k}\leq 1$, $0\leq f\leq 1$ and $\{f_{k}\}_{k\in\mathbbm{N}}$ is a monotone sequence. Therefore $f_{k}\to f$ in $L^{p}$ for any $p>1$ as $k\to \infty$ which implies
$$
\int_{X}\big(u-P_{\omega}[u]\big)MA_{\omega}(v_{\psi_{k}})\to \int_{X}\big(u-P_{\omega}[u]\big)MA_{\omega}(v_{\psi})
$$
as $k\to \infty$ since $MA_{\omega}(\varphi)$ is a volume form. Hence (\ref{eqn:Mab}) follows since by the first part of the proof
$$
\int_{X}\big(u_{k}-P_{\omega}[u_{k}]-u+P_{\omega}[u]\big)MA_{\omega}(v_{\psi_{k}})\to 0.
$$
\end{proof}
\section{Theorem \ref{thmA}}
\label{sec:Varia}
In this section we fix $\psi\in\mathcal{M}^{+}$ and using a variational approach we first prove the bijectivity of the Monge-Ampère operator between $\mathcal{E}^{1}_{norm}(X,\omega,\psi)$ and $\mathcal{M}^{1}(X,\omega,\psi)$, and then we prove that it is actually a homeomorphism considering the strong topologies.
\subsection{Degenerate complex Monge-Ampère equations.}
\label{ssec:Dege}
Letting $\mu$ be a probability measure and $\psi\in\mathcal{M}$, we define the functional $F_{\mu,\psi}:\mathcal{E}^{1}(X,\omega,\psi)\to \mathbbm{R}\cup\{-\infty\}$ as
$$
F_{\mu,\psi}(u):=(E_{\psi}-V_{\psi}L_{\mu})(u)
$$
where we recall that $L_{\mu}(u)=\lim_{j\to \infty}L_{\mu}\big(\max(u,\psi-j)\big)=\lim_{j\to \infty}\int_{X}\big(\max(u,\psi-j)-\psi\big)\mu$ (see section \ref{sec:ActionMeasure}). $F_{\mu,\psi}$ is clearly a translation invariant functional and $F_{\mu,\psi}\equiv 0$ for any $\mu$ if $V_{\psi}=0$.
\begin{prop}
\label{prop:Coerc}
Let $\mu$ be a probability measure, $\psi\in\mathcal{M}^{+}$ and let $F:=F_{\mu,\psi}$. If $L_{\mu}$ is continuous then $F$ is upper semicontinuous on $\mathcal{E}^{1}(X,\omega,\psi)$. Moreover if $L_{\mu}$ is finite valued on $\mathcal{E}^{1}(X,\omega,\psi)$ then there exist $A,B>0$ such that
$$
F(v)\leq -Ad(\psi,v)+B
$$
for any $v\in\mathcal{E}^{1}_{norm}(X,\omega,\psi)$, i.e. $F$ is $d$\emph{-coercive}. In particular $F$ is upper semicontinuous on $\mathcal{E}^{1}(X,\omega,\psi)$ and $d$-coercive on $\mathcal{E}^{1}_{norm}(X,\omega,\psi)$ if $\mu=MA_{\omega}(u)/V_{\psi}$ for $u\in\mathcal{E}^{1}(X,\omega,\psi)$.
\end{prop}
\begin{proof}
If $L_{\mu}$ is continuous then $F$ is easily upper semicontinuous by Proposition \ref{prop:PropertiesE}.\\
Then, since $d(\psi,v)=-E_{\psi}(v)$ on $\mathcal{E}^{1}_{norm}(X,\omega,\psi)$, it is easy to check that the coercivity requested is equivalent to
$$
\sup_{\mathcal{E}^{1}_{C}(X,\omega,\psi)\cap \mathcal{E}^{1}_{norm}(X,\omega,\psi)}|L_{\mu}|\leq \frac{(1-A)}{V_{\psi}} C+O(1),
$$
which holds by Proposition \ref{prop:Imp}.(ii).\\
Next assuming $\mu=MA_{\omega}(u)/V_{\psi}$ it is sufficient to check the continuity of $L_{\mu}$ since $L_{\mu}$ is finite valued on $\mathcal{E}^{1}(X,\omega,\psi)$ by Proposition \ref{prop:PropertiesL}. We may suppose without loss of generality that $u\leq \psi$. By Proposition \ref{prop:LastLast} and Remark 	\ref{rem:Usef}, for any $C\in\mathbbm{R}_{>0}$, $L_{\mu}$ restricted to $\mathcal{E}^{1}_{C}(X,\omega,\psi)$ is the uniform limit of $L_{\mu_{j}}$, where $\mu_{j}:=MA_{\omega}\big(\max(u,\psi-j)\big)$, since $I_{\psi}\big(\max(u,\psi-j),u\big)\to 0$ as $j\to \infty$. Therefore $L_{\mu}$ is continuous on $\mathcal{E}^{1}_{C}(X,\omega,\psi)$ since uniform limit of continuous functionals $L_{\mu_{j}}$ (Corollary \ref{cor:Useful}).
\end{proof}
As a consequence of the concavity of $E_{\psi}$ if $\mu=MA_{\omega}(u)/V_{\psi}$ for $u\in\mathcal{E}^{1}(X,\omega,\psi)$ where $V_{\psi}>0$ then
$$
J_{u}^{\psi}(\psi)=F_{\mu,\psi}(u)=\sup_{\mathcal{E}^{1}(X,\omega,\psi)}F_{\mu,\psi},
$$
i.e. $u$ is a maximizer for $F_{\mu,\psi}$. The other way around also holds as the next result shows.
\begin{prop}
\label{prop:IFF}
Let $\psi\in\mathcal{M}^{+}$ and let $\mu$ be a probability measure such that $L_{\mu}$ is finite valued on $\mathcal{E}^{1}(X,\omega,\psi)$. Then $\mu=MA_{\omega}(u)/V_{\psi}$ for $u\in\mathcal{E}^{1}(X,\omega,\psi)$ if and only if $u$ is a maximizer of $F_{\mu,\psi}$.
\end{prop}
\begin{proof}
As said before, it is clear that $\mu=MA_{\omega}(u)/V_{\psi}$ implies that $u$ is a maximizer for $F_{\mu,\psi}$. Conversely, if $u$ is a maximizer of $F_{\mu,\psi}$ then by Theorem $4.22$ in \cite{DDNL17b} $\mu=MA_{\omega}(u)/V_{\psi}$.
\end{proof}
Similarly to \cite{BBGZ09} we, thus, define the \emph{$\psi$-relative energy} for $\psi\in\mathcal{M}$ of a probability measure $\mu$ as
$$
E^{*}_{\psi}(\mu):=\sup_{u\in\mathcal{E}^{1}(X,\omega,\psi)}F_{\mu,\psi}(u)
$$
i.e. essentially as the Legendre trasform of $E_{\psi}$. It takes non-negative values ($F_{\mu,\psi}(\psi)=0$) and it is easy to check that $ E^{*}_{\psi}$ is a convex function.\\
Moreover defining
$$
\mathcal{M}^{1}(X,\omega,\psi):=\{V_{\psi}\mu \, :\, \mu\, \mbox{is a probability measure satisfying}\, E_{\psi}^{*}(\mu)<\infty\},
$$
we note that $\mathcal{M}^{1}(X,\omega,\psi)$ consists only of the null measure if $V_{\psi}=0$ while in $V_{\psi}>0$ any probability measure $\mu$ such that $V_{\psi}\mu\in\mathcal{M}^{1}(X,\omega,\psi)$ is non-pluripolar as the next Lemma shows.
\begin{lem}
\label{lem:Pluripolar}
Let $A\subset X$ be a (locally) pluripolar set. Then there exists $u\in\mathcal{E}^{1}(X,\omega,\psi)$ such that $A\subset\{u=-\infty\}$. In particular if $V_{\psi}\mu\in \mathcal{M}^{1}(X,\omega,\psi)$ for $\psi\in\mathcal{M}^{+}$ then $\mu$ is non-pluripolar.
\end{lem}
\begin{proof}
By Corollary $2.11$ in \cite{BBGZ09} there exists $\varphi\in\mathcal{E}^{1}(X,\omega)$ such that $A\subset \{\varphi=-\infty\}$. Therefore setting $u:=P_{\omega}[\psi](\varphi)$ proves the first part.\\
Next let $V_{\psi}\mu\in\mathcal{M}^{1}(X,\omega,\psi)$ for $\psi\in\mathcal{M}^{+}$ and $\mu$ probability measure and assume by contradiction that $\mu$ takes mass on a pluripolar set $A$. Then by the first part of the proof there exists $u\in \mathcal{E}^{1}(X,\omega,\psi)$ such that $A\subset \{u=-\infty\}$. On the other hand, since $V_{\psi}\mu\in\mathcal{M}^{1}(X,\omega,\psi)$ by definition $\mu$ does not charge $\{\psi=-\infty\}$. Thus by Proposition \ref{prop:PropertiesL}$.(iii)$ we obtain $L_{\mu}(u)=-\infty$, which is a contradiction.
\end{proof}
We can now prove that the Monge-Ampère operation is a bijection between $\mathcal{E}^{1}(X,\omega,\psi)$ and $\mathcal{M}^{1}(X,\omega,\psi)$.
\begin{lem}
\label{lem:CapacityOk}
Let $\psi\in\mathcal{M}^{+}$ and let $\mu\in\mathcal{C}_{A,\psi}$ where $A\in\mathbbm{R}$. Then there exists $u\in\mathcal{E}^{1}_{norm}(X,\omega,\psi)$ maximizing $F_{\mu,\psi}$.
\end{lem}
\begin{proof}
By Lemma \ref{lem:4.18} $L_{\mu}$ is finite valued on $\mathcal{E}^{1}(X,\omega,\psi)$, and it is continuous on $\mathcal{E}^{1}_{C}(X,\omega,\psi)$ for any $C\in\mathbbm{R}$ thanks to Corollary \ref{cor:Useful}. Therefore it follows from Proposition \ref{prop:Coerc} that $F_{\mu,\psi}$ is upper semicontinuous and $d$-coercive on $\mathcal{E}^{1}_{norm}(X,\omega,\psi)$. Hence $F_{\mu,\psi}$ admits a maximizer $u\in\mathcal{E}^{1}_{norm}(X,\omega,\psi)$ as easy consequence of the weak compactness of $\mathcal{E}^{1}_{C}(X,\omega,\psi)$.
\end{proof}
\begin{prop}
\label{prop:OldPropC}
Let $\psi\in\mathcal{M}^{+}$. Then the Monge-Ampère map $MA:\mathcal{E}^{1}_{norm}(X,\omega,\psi)\to \mathcal{M}^{1}(X,\omega,\psi)$, $u\to MA(u)$ is bijective. Furthermore if $V_{\psi}\mu=MA_{\omega}(u)\in\mathcal{M}^{1}(X,\omega,\psi)$ for $u\in\mathcal{E}^{1}(X,\omega,\psi)$ then any maximizing sequence $u_{k}\in\mathcal{E}^{1}_{norm}(X,\omega,\psi)$ for $F_{\mu,\psi}$ necessarily converges weakly to $u$.
\end{prop}
\begin{proof}
The proof is inspired by Theorem $4.7$ in \cite{BBGZ09}.\\
The map is well-defined as a consequence of Proposition \ref{prop:Coerc}, i.e. $MA_{\omega}(u)\in\mathcal{M}^{1}(X,\omega,\psi)$ for any $u\in\mathcal{E}^{1}(X,\omega,\psi)$. Moreover the injectivity follows from Theorem $4.8$ in \cite{DDNL18b}.\\
Let $u_{k}\in\mathcal{E}^{1}_{norm}(X,\omega,\psi)$ be a sequence such that $F_{\mu,\psi}(u_{k})\nearrow \sup_{\mathcal{E}^{1}(X,\omega,\psi)}F_{\mu,\psi}$ where $\mu=MA_{\omega}(u)/V_{\psi}$ is a probability measure and $u\in\mathcal{E}^{1}_{norm}(X,\omega,\psi)$. Up to considering a subsequence, we may also assume that $u_{k}\to v\in PSH(X,\omega)$. Then, by the upper semicontinuity and the $d$-coercivity of $F_{\mu,\psi}$ (Proposition \ref{prop:Coerc}) it follows that $v\in\mathcal{E}^{1}_{norm}(X,\omega,\psi)$ and $F_{\mu,\psi}(v)=\sup_{\mathcal{E}^{1}(X,\omega,\psi)}F_{\mu,\psi}$. Thus by Proposition \ref{prop:IFF} we get $\mu=MA_{\omega}(v)/V_{\psi}$. Hence $v=u$ since $\sup_{X}v=\sup_{X}u=0$.\\
Then let $\mu$ be a probability measure such that $V_{\psi}\mu\in\mathcal{M}^{1}(X,\omega,\psi)$. Again by Proposition \ref{prop:IFF}, to prove the existence of $u\in\mathcal{E}^{1}_{norm}(X,\omega,\psi)$ such that $\mu=MA_{\omega}(u)/V_{\psi}$ it is sufficient to check that $F_{\mu,\psi}$ admits a maximum over $\mathcal{E}^{1}_{norm}(X,\omega,\psi)$. Moreover by Proposition \ref{prop:Coerc} we also know that $F_{\mu,\psi}$ is $d$-coercive on $\mathcal{E}^{1}_{norm}(X,\omega,\psi)$. Thus if there exists a constant $A>0$ such that $\mu\in\mathcal{C}_{A,\psi}$ then Corollary \ref{cor:Useful} leads to the upper semicontinuity of $F_{\mu,\psi}$ which clearly implies that $V_{\psi}\mu=MA_{\omega}(u)$ for $u\in\mathcal{E}^{1}(X,\omega,\psi)$ since $\mathcal{E}^{1}_{C}(X,\omega,\psi)\subset PSH(X,\omega)$ is compact for any $C\in\mathbbm{R}_{>0}$. \\
In the general case by Lemma $4.26$ in \cite{DDNL17b} (see also \cite{Ceg98}) $\mu$ is absolutely continuous with respect to $\nu\in\mathcal{C}_{1,\psi}$ using also that $\mu$ is a non-pluripolar measure (Lemma \ref{lem:Pluripolar}). Therefore letting $f\in L^{1}(\nu)$ such that $\mu=f\nu$, we define for any $k\in\mathbbm{N}$
$$
\mu_{k}:=(1+\epsilon_{k})\min(f,k)\nu
$$
where $\epsilon_{k}>0$ are chosen so that $\mu_{k}$ is a probability measure, noting that $(1+\epsilon_{k})\min(f,k)\to f$ in $L^{1}(\nu)$. Then by Lemma \ref{lem:CapacityOk} it follows that $\mu_{k}=MA_{\omega}(u_{k})/V_{\psi}$ for $u_{k}\in\mathcal{E}^{1}_{norm}(X,\omega,\psi)$.\\
Moreover by weak compactness, without loss of generality, we may also assume that $u_{k}\to u\in PSH(X,\omega)$. Note that $u\leq \psi$ since $u_{k}\leq \psi$ for any $k\in\mathbbm{N}$. Then by Lemma $2.8$ in \cite{DDNL18b} we obtain
$$
MA_{\omega}(u)\geq V_{\psi}f\nu=V_{\psi}\mu,
$$
which implies $MA_{\omega}(u)=V_{\psi}\mu $ by \cite{WN17} since $u$ is more singular than $\psi$ and $\mu$ is a probability measure. It remains to prove that $u\in\mathcal{E}^{1}(X,\omega,\psi)$.\\
It is not difficult to see that $\mu_{k}\leq 2\mu$ for $k\gg 0$, thus Proposition \ref{prop:Imp} implies that there exists a constant $B>0$ such that
$$
\sup_{\mathcal{E}^{1}_{C}(X,\omega,\psi)}|L_{\mu_{k}}|\leq 2\sup_{\mathcal{E}^{1}_{C}(X,\omega,\psi)}|L_{\mu}|\leq 2B(1+C^{1/2})
$$
for any $C\in\mathbbm{R}_{>0}$. Therefore
$$
J^{\psi}_{u_{k}}(\psi)=E_{\psi}(u_{k})+V_{\psi}|L_{\mu_{k}}(u_{k})|\leq \sup_{C>0}\big( 2V_{\psi}B(1+C^{1/2})-C\big)
$$
and Lemma \ref{lem:Related} yields $d(\psi,u_{k})\leq D $ for a uniform constant $D$, i.e. $u_{k}\in\mathcal{E}^{1}_{D'}(X,\omega,\psi)$ for any $k\in\mathbbm{N}$ for a uniform constant $D'$ (Remark \ref{rem:Usef}). Hence since $\mathcal{E}^{1}_{D'}(X,\omega,\psi)$ is weakly compact we obtain $u\in\mathcal{E}^{1}_{D'}(X,\omega,\psi)$. 
\end{proof}
\subsection{Proof of Theorem \ref{thmA}.}
We first need to explore further the properties of the strong topology on $\mathcal{E}^{1}(X,\omega,\psi)$.\\

By Proposition \ref{prop:L1} the strong convergence implies the weak convergence. Moreover the strong topology is the coarsest refinement of the weak topology such that $E_{\psi}(\cdot)$ becomes continuous.
\begin{prop}
\label{prop:PropAOne}
Let $\psi\in\mathcal{M}^{+} $ and $u_{k},u\in\mathcal{E}^{1}(X,\omega,\psi)$. Then $u_{k}\to u$ strongly if and only if $u_{k}\to u$ weakly and $E_{\psi}(u_{k})\to E_{\psi}(u)$.
\end{prop}
\begin{proof}
Assume that $u_{k}\to u$ weakly and that $E_{\psi}(u_{k})\to E_{\psi}(u)$. Then $w_{k}:=(\sup\{u_{j}\, : \, j\geq k\}\big)^{*}\in \mathcal{E}^{1}(X,\omega,\psi)$ and it decreases to $u$. Thus by Proposition \ref{prop:PropertiesE} $E_{\psi}(w_{k})\to E_{\psi}(u)$ and
$$
d(u_{k},u)\leq d(u_{k},w_{k})+d(w_{k},u)=2E_{\psi}(w_{k})-E_{\psi}(u_{k})-E_{\psi}(u)\to 0.
$$
Conversely, assuming that $d(u_{k},u)\to 0 $, we immediately get that $u_{k}\to u$ weakly as said above (Proposition \ref{prop:L1}). Moreover $\sup_{X}u_{k},\sup_{X}u\leq A$ uniformly for a constant $A\in\mathbbm{R}$. Thus
$$
|E_{\psi}(u_{k})-E_{\psi}(u)|=|d(\psi+A,u_{k})-d(\psi+A,u)|\leq d(u_{k},u)\to 0,
$$
which concludes the proof.
\end{proof}
Then we also observe that the strong convergence implies the convergence in $\psi'$-capacity for any $\psi'\in\mathcal{M}^{+}$.
\begin{prop}
\label{prop:CapacityPsi}
Let $\psi\in\mathcal{M}^{+}$ and $u_{k},u\in\mathcal{E}^{1}(X,\omega,\psi)$ such that $d(u_{k},u)\to 0$. Then there exists a subsequence $\{u_{k_{j}}\}_{j\in\mathbbm{N}}$ such that $w_{j}:=\big(\sup\{u_{k_{h}}\, : \, h\geq j\}\big)^{*}$, $v_{j}:=P_{\omega}(u_{k_{j}},u_{k_{j+1}},\dots)$ belong to $\mathcal{E}^{1}(X,\omega,\psi)$ and converge monotonically almost everywhere to $u$. In particular $u_{k}\to u$ in $\psi'$-capacity for any $\psi'\in\mathcal{M}^{+}$ and $MA_{\omega}(u_{k}^{j},\psi^{n-j})\to MA_{\omega}(u^{j},\psi^{n-j})$ weakly for any $j=0,\dots,n$.
\end{prop}
\begin{proof}
Since the strong convergence implies the weak convergence by Proposition \ref{prop:PropAOne} it is clear that $w_{k}\in\mathcal{E}^{1}(X,\omega,\psi)$ and that it decreases to $u$. In particular up to considering a subsequence we may assume that $d(u_{k},w_{k})\leq 1/2^{k}$ for any $k\in\mathbbm{N}$.\\
Next for any $j\geq k$ we set $v_{k,j}:=P_{\omega}(u_{k},\dots,u_{j})\in\mathcal{E}^{1}(X,\omega,\psi)$ and $v_{k,j}^{u}:=P_{\omega}(v_{k,j},u)\in\mathcal{E}^{1}(X,\omega,\psi)$. Then it follows from Proposition \ref{prop:PropertiesE} and Lemma $3.7$ in \cite{DDNL17b} that
\begin{multline*}
d(u,v_{k,j}^{u})\leq \int_{X}(u-v_{k,j}^{u})MA_{\omega}(v_{k,j}^{u})\leq \int_{\{v_{k,j}^{u}=v_{k,j}\}}(u-v_{k,j})MA_{\omega}(v_{k,j})\leq\\
\leq\sum_{s=k}^{j}\int_{X}(w_{s}-u_{s})MA_{\omega}(u_{s})\leq (n+1)\sum_{s=k}^{j}d(w_{s},u_{s})\leq \frac{(n+1)}{2^{k-1}}.
\end{multline*}
Therefore by Proposition \ref{prop:CompactL1} $v_{k,j}^{u}$ decreases (hence converges strongly) to a function $\phi_{k}\in\mathcal{E}^{1}(X,\omega,\psi)$ as $j\to \infty$. Similarly we also observe that
$$
d(v_{k,j},v_{k,j}^{u})\leq \int_{\{v_{k,j}^{u}=u\}}(v_{k,j}-u)MA_{\omega}(u)\leq \int_{X}|v_{k,1}-u|MA_{\omega}(u)\leq C
$$
uniformly in $j$ by Corollary \ref{cor:ForL1}. Hence by definition $d(u,v_{k,j})\leq C+\frac{(n+1)}{2^{k-1}}$, i.e. $v_{k,j}$ decreases and converges strongly as $j\to \infty$ to the function $v_{k}=P_{\omega}(u_{k},u_{k+1}\dots)\in\mathcal{E}^{1}(X,\omega,\psi)$ again by Proposition \ref{prop:CompactL1}. Moreover by construction $u_{k}\geq v_{k}\geq \phi_{k}$ since $v_{k}\leq v_{k,j}\leq u_{k}$ for any $j\geq k$. Hence
$$
d(u,v_{k})\leq d(u,\phi_{k})\leq \frac{(n+1)}{2^{k-1}}\to 0
$$
as $k\to \infty$, i.e. $v_{k}\nearrow u$ strongly.\\
The convergence in $\psi'$-capacity for $\psi'\in\mathcal{M}^{+}$ in now clearly an immediate consequence. Indeed by an easy contradiction argument it is enough to prove that any arbitrary subsequence, which we will keep denoting with $\{u_{k}\}_{k\in\mathbbm{N}}$ for the sake of simplicity, admits a further subsequence $\{u_{k_{j}}\}_{j\in\mathbbm{N}}$ converging in $\psi'$-capacity to $u$. Thus taking the subsequence satisfying $v_{j}\leq u_{k_{j}}\leq w_{j}$ where $v_{j},w_{j}$ are the monotonic sequence of the first part of the Proposition, the convergence in $\psi'$-capacity follows from the inclusions
\begin{gather*}
\{|u-u_{k_{j}}|>\delta\}=\{u-u_{k_{j}}>\delta\}\cup\{u_{k_{j}}-u>\delta\}\subset \{u-v_{j}>\delta\}\cup \{w_{j}-u>\delta\}
\end{gather*}
for any $\delta>0$. Finally Lemma \ref{lem:KeyConv} gives the weak convergence of the measures.
\end{proof}
We can now endow the set $\mathcal{M}^{1}(X,\omega,\psi)=\{V_{\psi}\mu\, :\, \mu \, \mbox{is a probability measure satisfying} \, E^{*}_{\psi}(\mu)<+\infty\}$ (subsection \ref{ssec:Dege}) with its natural strong topology given as the coarsest refinement of the weak topology such that $E_{\psi}^{*}(\cdot)$ becomes continuous, and prove our Theorem \ref{thmA}.
\begin{reptheorem}{thmA}
Let $\psi\in\mathcal{M}^{+}$. Then
$$
MA_{\omega}:\big(\mathcal{E}^{1}_{norm}(X,\omega,\psi),d\big)\to \big(\mathcal{M}^{1}(X,\omega,\psi),strong\big)
$$
is a homeomorphism.
\end{reptheorem}
\begin{proof}
The map is bijective as immediate consequence of Proposition \ref{prop:OldPropC}.\\
Next, letting $u_{k}\in\mathcal{E}^{1}_{norm}(X,\omega,\psi)$ converging strongly to $u\in\mathcal{E}^{1}_{norm}(X,\omega,\psi)$, Proposition \ref{prop:CapacityPsi} gives the weak convergence of $MA_{\omega}(u_{k})\to MA_{\omega}(u)$ as $k\to \infty$. Moreover since $E_{\psi}^{*}\big(MA_{\omega}(v)/V_{\psi}\big)=J_{v}^{\psi}(\psi)$ for any $v\in\mathcal{E}^{1}(X,\omega,\psi)$, we get
\begin{multline}
\label{eqn:ThmA1}
\Big|E^{*}_{\psi}\big(MA_{\omega}(u_{k})/V_{\psi}\big)-E_{\psi}^{*}\big(MA_{\omega}(u)/V_{\psi}\big)\Big|\leq \\
\leq\big|E_{\psi}(u_{k})-E_{\psi}(u)\big|+\Big|\int_{X}(\psi-u_{k})MA_{\omega}(u_{k})-\int_{X}(\psi-u)MA_{\omega}(u)\Big|\leq\\
\leq \big|E_{\psi}(u_{k})-E_{\psi}(u)\big|+\Big|\int_{X}(\psi-u_{k})\big(MA_{\omega}(u_{k})-MA_{\omega}(u)\big)\Big|+\int_{X}|u_{k}-u|MA_{\omega}(u).
\end{multline}
Hence $MA_{\omega}(u_{k})\to MA_{\omega}(u)$ strongly in $\mathcal{M}^{1}(X,\omega,\psi)$ since each term on the right-hand side of (\ref{eqn:ThmA1}) goes to $0$ as $k\to +\infty$ combining Proposition \ref{prop:PropAOne}, Proposition \ref{prop:LastLast} and Corollary \ref{cor:ForL1} recalling that by Proposition \ref{prop:ForL1} $I_{\psi}(u_{k},u)\to 0$ as $k\to \infty$.\\
Conversely, suppose that $MA_{\omega}(u_{k})\to MA_{\omega}(u)$ strongly in $\mathcal{M}^{1}(X,\omega,\psi)$ where $u_{k},u\in\mathcal{E}^{1}_{norm}(X,\omega,\psi)$. Then, letting $\{\varphi_{j}\}_{j\in\mathbbm{N}}\subset \mathcal{H}_{\omega}$ such that $\varphi_{j}\searrow u$ (\cite{BK07}) and setting $v_{j}:=P_{\omega}[\psi](\varphi_{j})$, by Lemma \ref{lem:Related}
\begin{multline}
(n+1)I_{\psi}(u_{k},v_{j})\leq E_{\psi}(u_{k})-E_{\psi}(v_{j})+\int_{X}(v_{j}-u_{k})MA_{\omega}(u_{k})=\\
=E_{\psi}^{*}\big(MA_{\omega}(u_{k})/V_{\psi}\big)-E_{\psi}^{*}\big(MA_{\omega}(v_{j})/V_{\psi}\big)+\int_{X}(v_{j}-\psi)\big(MA_{\omega}(u_{k})-MA_{\omega}(v_{j})\big).
\end{multline}
By construction and the first part of the proof, it follows that $E^{*}_{\psi}\big(MA_{\omega}(u_{k})/V_{\psi}\big)-E_{\psi}^{*}\big(MA_{\omega}(v_{j})/V_{\psi}\big)\to 0$ as $k,j\to \infty$. While setting $f_{j}:=v_{j}-\psi$ we want to prove that
$$
\limsup_{k\to \infty}\int_{X}f_{j}MA_{\omega}(u_{k})=\int_{X}f_{j}MA_{\omega}(u),
$$
which would imply $\limsup_{j\to \infty}\limsup_{k\to \infty}I_{\psi}(u_{k},v_{j})=0$ since $\int_{X}f_{j}\big(MA_{\omega}(u)-MA_{\omega}(v_{j})\big)\to 0$ as a consequence of Propositions \ref{prop:LastLast} and \ref{prop:ForL1}.\\
We observe that $||f_{j}||_{L^{\infty}}\leq ||\varphi_{j}||_{L^{\infty}}$ by Proposition \ref{prop:PropProie} and we denote by $\{f_{j}^{s}\}_{s\in\mathbbm{N}}\subset C^{\infty}$ a sequence of smooth functions converging in capacity to $f_{j}$ such that $||f_{j}^{s}||_{L^{\infty}}\leq 2||f_{j}||_{L^{\infty}}$. We recall here briefly how to construct such sequence. Let $\{g_{j}^{s}\}_{s\in\mathbbm{N}}$ be the sequence of bounded functions converging in capacity to $f_{j}$ defined as $g_{j}^{s}:=\max(v_{j},-s)-\max(\psi,-s)$. We have that $||g_{j}^{s}||_{L^{\infty}}\leq ||f_{j}||_{L^{\infty}}$ and that $\max(v_{j},-s), \max(\psi,-s)\in PSH(X,\omega)$. Therefore by a regularization process (see for instance \cite{BK07}) and a diagonal argument we can now construct a sequence $\{f_{j}^{s}\}_{j\in\mathbbm{N}}\subset C^{\infty}$ converging in capacity to $f_{j}$ such that $||f_{j}^{s}||_{L^{\infty}}\leq2||g_{j}^{s}||\leq 2||f_{j}||_{L^{\infty}}$ where $f_{j}^{s}=v_{j}^{s}-\psi^{s}$ with $v_{j}^{s},\psi^{s}$ quasi-psh functions decreasing respectively to $v_{j},\psi$.\\
Then letting $\delta>0$ we have
$$
\int_{X}(f_{j}-f_{j}^{s})MA_{\omega}(u_{k})\leq \delta V_{\psi}+3||\varphi_{j}||_{L^{\infty}}\int_{\{f_{j}-f_{j}^{s}>\delta\}}MA_{\omega}(u_{k})\leq \delta V_{\psi}+3||\varphi_{j}||_{L^{\infty}}\int_{\{\psi^{s}-\psi>\delta\}}MA_{\omega}(u_{k})
$$
from the trivial inclusion $\{f_{j}-f_{j}^{s}>\delta\}\subset \{\psi^{s}-\psi>\delta\}$. Therefore
\begin{multline*}
\limsup_{s\to \infty}\limsup_{k\to \infty}\int_{X}(f_{j}-f_{j}^{s})MA_{\omega}(u_{k})\leq \delta V_{\psi}+\limsup_{s\to \infty}\limsup_{k\to \infty}\int_{\{\psi^{s}-\psi\geq \delta\}}MA_{\omega}(u_{k})\leq\\
\leq \delta V_{\psi}+\limsup_{s\to \infty}\int_{\{\psi^{s}-\psi\geq \delta\}}MA_{\omega}(u)=\delta V_{\psi},
\end{multline*}
where we used that $\{\psi^{s}-\psi\geq \delta\}$ is a closed set in the plurifine topology. Hence since $f_{j}^{s}\in C^{\infty}$ we obtain
\begin{multline*}
\limsup_{k\to \infty}\int_{X}f_{j}MA_{\omega}(u_{k})=\limsup_{s\to \infty}\limsup_{k\to \infty}\Big(\int_{X}(f_{j}-f_{j}^{s})MA_{\omega}(u_{k})+\int_{X}f_{j}^{s}MA_{\omega}(u_{k})\Big)\leq\\
\leq \limsup_{s\to \infty}\int_{X}f_{j}^{s}MA_{\omega}(u)=\int_{X}f_{j}MA_{\omega}(u),
\end{multline*}
which as said above implies $I_{\psi}(u_{k},v_{j})\to 0$ letting $k,j\to \infty$ in this order. \\
Next, again by Lemma \ref{lem:Related}, we obtain $u_{k}\in\mathcal{E}^{1}_{C}(X,\omega,\psi)$ for some $C\in\mathbbm{N}$ big enough since $J_{u_{k}}^{\psi}(\psi)=E^{*}_{\psi}\big(MA_{\omega}(u_{k})/V_{\psi}\big)$. In particular, up to considering a subsequence, $u_{k}\to w\in\mathcal{E}^{1}_{norm}(X,\omega,\psi)$ weakly by Proposition \ref{prop:CompactL1}. Observe also that by Proposition \ref{prop:LastLast}
\begin{equation}
\label{eqn:CP1}
\Big|\int_{X}(\psi-u_{k})\big(MA_{\omega}(v_{j})-MA_{\omega}(u_{k})\big)\Big|\to 0
\end{equation}
as $k,j\to \infty$ in this order. Moreover by Proposition \ref{prop:USC} and Lemma \ref{lem:Equicontinuity}
\begin{multline}
\label{eqn:CP2}
\limsup_{k\to \infty}\Big(E^{*}_{\psi}\big(MA_{\omega}(u_{k})/V_{\psi}\big)+\int_{X}(\psi-u_{k})\big(MA_{\omega}(v_{j})-MA_{\omega}(u_{k})\big)\Big)=\\
=\limsup_{k\to \infty}\Big(E_{\psi}(u_{k})+\int_{X}(\psi-u_{k})MA_{\omega}(v_{j})\Big)\leq E_{\psi}(w)+\int_{X}(\psi-w)MA_{\omega}(v_{j}).
\end{multline}
Therefore combining (\ref{eqn:CP1}) and (\ref{eqn:CP2}) with the strong convergence of $v_{j}$ to $u$ we obtain
\begin{multline*}
E_{\psi}(u)+\int_{X}(\psi-u)MA_{\omega}(u)=\lim_{k\to \infty}E^{*}_{\psi}\big(MA_{\omega}(u_{k})/V_{\psi}\big)\leq\\
\leq \limsup_{j\to \infty}\Big(E_{\psi}(w)+\int_{X}(\psi-w)MA_{\omega}(v_{j})\Big)=E_{\psi}(w)+\int_{X}(\psi-w)MA_{\omega}(u),
\end{multline*}
i.e. $w$ is a maximizer of $F_{MA_{\omega}(u)/V_{\psi},\psi}$. Hence $w=u$ (Proposition \ref{prop:OldPropC}), i.e. $u_{k}\to u$ weakly. Furthermore again by Lemma \ref{lem:Related} and Lemma \ref{lem:Equicontinuity}
\begin{multline}
\limsup_{k\to\infty}\big(E_{\psi}(v_{j})-E_{\psi}(u_{k})\big)\leq \limsup_{k\to \infty}\Big(\frac{n}{n+1}I_{\psi}(u_{k},v_{j})+\Big|\int_{X}(u_{k}-v_{j})MA_{\omega}(v_{j})\Big|\Big)\leq\\
\leq\Big|\int_{X}(u-v_{j})MA_{\omega}(v_{j})\Big|+\limsup_{k\to \infty}\frac{n}{n+1}I_{\psi}(u_{k},v_{j}).
\end{multline}
Finally letting $j\to \infty$, since $v_{j}\searrow u$ strongly, we obtain $\liminf_{j\to \infty}E_{\psi}(u_{k})\geq\lim_{j\to \infty} E_{\psi}(v_{j})= E_{\psi}(u)$ which implies that $E_{\psi}(u_{k})\to E_{\psi}(u)$ and that $u_{k}\to u$ strongly by Proposition \ref{prop:PropAOne}. 
\end{proof}
The main difference between the proof of Theorem \ref{thmA} with respect to the same result in the absolute setting, i.e. when $\psi=0$, is that for fixed $u\in\mathcal{E}^{1}(X,\omega,\psi)$ the action $\mathcal{M}^{1}(X,\omega,\psi)\ni MA_{\omega}(v)\to \int_{X}(u-\psi)MA_{\omega}(v)$ is not a priori continuous with respect to the weak topologies of measures even if we restrict the action on $\mathcal{M}^{1}_{C}(X,\omega,\psi):=\{V_{\psi}\mu\, : \, E^{*}_{\psi}(\mu)\leq C\}$ for $C\in\mathbbm{R}$ while in the absolute setting this is given by Proposition $1.7.$ in \cite{BBEGZ16} where the authors used the fact that any $u\in\mathcal{E}^{1}(X,\omega)$ can be approximated inside the class $\mathcal{E}^{1}(X,\omega)$ by a sequence of continuous functions.
\section{Strong Topologies.}
\label{sec:Strong}
In this section we investigate the strong topology on $X_{\mathcal{A}}$ in detail, proving that it is the coarsest refinement of the weak topology such that $E_{\cdot}(\cdot)$ becomes continuous (Theorem \ref{thm:OldPropA}) and proving that the strong convergence implies the convergence in $\psi$-capacity for any $\psi\in\mathcal{M}^{+}$ (Theorem \ref{thm:OldPropB}), i.e. we extend all the typical properties of the $L^{1}$-metric geometry to the bigger space $X_{\mathcal{A}}$, justifying further the construction of the distance $d_{\mathcal{A}}$ (\cite{Tru19}) and its naturality. Moreover we define the set $Y_{\mathcal{A}}$, and we prove Theorem \ref{thmB}.
\subsection{About $\big(X_{\mathcal{A}},d_{\mathcal{A}}\big)$.}
First we prove that the strong convergence in $X_{\mathcal{A}}$ implies the weak convergence, recalling that for weak convergence of $u_{k}\in \mathcal{E}^{1}(X,\omega,\psi_{k})$ to $P_{\psi_{\min}}$ where $\psi_{\min}\in \mathcal{M}$ with $V_{\psi_{\min}}=0$ we mean that $|\sup_{X}u_{k}|\leq C$ and that any weak accumulation point of $\{u_{k}\}_{k\in\mathbbm{N}}$ is more singular than $\psi_{\min}$.
\begin{prop}
\label{prop:StrongImpliesWeak}
Let $u_{k},u\in X_{\mathcal{A}}$ such that $u_{k}\to u$ strongly. If $u\neq P_{\psi_{\min}}$ then $u_{k}\to u$ weakly. If instead $u=P_{\psi_{\min}}$ the following dichotomy holds:
\begin{itemize}
\item[(i)] $u_{k}\to P_{\psi_{\min}}$ weakly;
\item[(ii)] $\limsup_{k\to \infty}|\sup_{X}u_{k}|=+\infty$.
\end{itemize}
\end{prop}
\begin{proof}
The dichotomy for the case $u=P_{\psi_{\min}}$ follows by definition. Indeed if $|\sup_{X}u_{k}|\leq C$ and $d_{\mathcal{A}}(u_{k},u)\to 0$ as $k\to \infty$, then
$V_{\psi_{k}}\to V_{\psi_{\min}}=0$ by Proposition \ref{prop:AllNecessary}$.(iv)$ which implies that $\psi_{k}\to \psi_{\min}$ by Lemma \ref{lem:HomeoV}. Hence any weak accumulation point $u$ of $\{u_{k}\}_{k\in\mathbbm{N}}$ satisfies $u\leq \psi_{\min}+C$.\\
Thus, let $\psi_{k},\psi\in\mathcal{A}$ such that $u_{k}\in\mathcal{E}^{1}(X,\omega,\psi_{k})$ and $u\in\mathcal{E}^{1}(X,\omega,\psi)$ where $\psi\in\mathcal{M}^{+}$. Observe that
$$
d(u_{k},\psi_{k})\leq d_{\mathcal{A}}(u_{k},u)+d(u,\psi)+d_{\mathcal{A}}(\psi,\psi_{k})\leq A
$$
for a uniform constant $A>0$ by Proposition \ref{prop:AllNecessary}$.(iv)$ \\ 
On the other hand for any $j\in\mathbbm{N}$ by \cite{BK07} there exists $h_{j}\in\mathcal{H}_{\omega}$ such that $h_{j}\geq u$, $||h_{j}-u||_{L^{1}}\leq 1/j$ and $d\big(u,P_{\omega}[\psi](h_{j})\big)\leq 1/j$. In particular by the triangle inequality and Proposition \ref{prop:AllNecessary} we have
\begin{equation}
\label{eqn:1Cha}
\limsup_{k\to \infty}d\big(P_{\omega}[\psi_{k}](h_{j}),\psi_{k}\big)\leq \limsup_{k\to \infty}\Big(d_{\mathcal{A}}\big(P_{\omega}[\psi_{k}](h_{j}),P_{\omega}[\psi](h_{j})\big)+\frac{1}{j}+d(u,\psi)+d(\psi,\psi_{k})\Big)\leq d(u,\psi)+\frac{1}{j},
\end{equation}
Similarly again by the triangle inequality and Proposition \ref{prop:AllNecessary}
\begin{equation}
\label{eqn:2Cha}
\limsup_{k\to \infty}d\big(u_{k},P_{\omega}[\psi_{k}](h_{j})\big)\leq\limsup_{k\to \infty}\Big(d_{\mathcal{A}}\big(P_{\omega}[\psi_{k}](h_{j}),P_{\omega}[\psi](h_{j})\big)+\frac{1}{j}+d_{\mathcal{A}}(u,u_{k})\Big)\leq \frac{1}{j}
\end{equation}
and
\begin{multline}
\label{eqn:KLM}
\limsup_{k\to \infty}||u_{k}-u||_{L^{1}}\leq\limsup_{k\to \infty}\Big( ||u_{k}-P_{\omega}[\psi_{k}](h_{j})||_{L^{1}}+||P_{\omega}[\psi_{k}](h_{j})-P_{\omega}[\psi](h_{j})||_{L^{1}}+||P_{\omega}[\psi](h_{j})-u||_{L^{1}}\Big)\leq\\
\leq\frac{1}{j}+\limsup_{k\to \infty}||u_{k}-P_{\omega}[\psi_{k}](h_{j})||_{L^{1}}
\end{multline}
where we also used Lemma \ref{lem:Referee}. In particular from (\ref{eqn:1Cha}) and (\ref{eqn:2Cha}) we deduce that $d\big(\psi_{k},P_{\omega}[\psi_{k}](h_{j})\big), d(\psi_{k},u_{k})\leq C$ for a uniform constant $C\in\mathbbm{R}$. Next let $\phi_{k}\in\mathcal{E}^{1}_{norm}(X,\omega,\psi)$ the unique solution of $MA_{\omega}(\phi_{k})=\frac{V_{\psi_{k}}}{V_{0}}MA_{\omega}(0)$ and observe that by Proposition \ref{prop:PropertiesE}
$$
d(\psi_{k},\phi_{k})=-E_{\psi_{k}}(\phi_{k})\leq \int_{X}(\psi_{k}-\phi_{k})MA_{\omega}(\phi_{k})\leq \frac{V_{\psi_{k}}}{V_{0}}\int_{X}|\phi_{k}|MA_{\omega}(0)\leq ||\phi_{k}||_{L^{1}}\leq C'
$$
since $\phi_{k}$ belongs to a compact (hence bounded) subset of $PSH(X,\omega)\subset L^{1}$. Therefore, since $V_{\psi_{k}}\geq a>0$ for $k\gg 0$ big enough, by Proposition \ref{prop:L1} it follows that there exists a continuous increasing function $f:\mathbbm{R}_{\geq 0}\to \mathbbm{R}_{\geq 0}$ with $f(0)=0$ such that
$$
||u_{k}-P_{\omega}[\psi_{k}](h_{j})||_{L^{1}}\leq f\big(d(u_{k},P_{\omega}[\psi_{k}](h_{j}))\big)
$$
for any $k,j$ big enough. Hence combining $(\ref{eqn:2Cha})$ and $(\ref{eqn:KLM})$ the convergence requested follows letting $k,j\to +\infty$ in this order.
\end{proof}
We can now prove the important characterization of the strong convergence as the coarsest refinement of the weak topology such that $E_{\cdot}(\cdot)$ becomes continuous.
\begin{thm}
\label{thm:OldPropA}
Let $u_{k}\in \mathcal{E}^{1}(X,\omega,\psi_{k}),u\in\mathcal{E}^{1}(X,\omega,\psi)$ for $\{\psi_{k}\}_{k\in\mathbbm{N}},\psi\in\overline{\mathcal{A}}$. If $\psi\neq \psi_{\min}$ or $V_{\psi_{\min}}>0$ then the followings are equivalent:
\begin{itemize}
\item[i)] $u_{k}\to u$ strongly;
\item[ii)] $u_{k}\to u$ weakly and $E_{\psi_{k}}(u_{k})\to E_{\psi}(u)$. 
\end{itemize}
In the case $\psi=\psi_{\min}$ and $V_{\psi_{\min}}=0$, if $u_{k}\to P_{\psi_{\min}}$ weakly and $E_{\psi_{k}}(u_{k})\to 0$ then $u_{k}\to P_{\psi_{\min}}$ strongly. Finally if $d_{\mathcal{A}}(u_{k},P_{\psi_{\min}})\to 0$ as $k\to \infty$, then the following dichotomy holds:
\begin{itemize}
\item[a)] $u_{k}\to P_{\psi_{\min}}$ weakly and $E_{\psi_{k}}(u_{k})\to 0$;
\item[b)] $\limsup_{k\to \infty}|\sup_{X}u_{k}|= \infty$.
\end{itemize}
\end{thm}
\begin{proof}
\textbf{Implication $\mathbf{(ii)\Rightarrow (i)}$.\\}
Assume that $(ii)$ holds where we include the case $u=P_{\psi_{\min}}$ setting $E_{\psi}(P_{\psi_{\min}}):=0$. Clearly it is enough to prove that any subsequence of $\{u_{k}\}_{k\in\mathbbm{N}}$ admits a subsequence which is $d_{\mathcal{A}}-$convergent to $u$. For the sake of simplicity we denote by $\{u_{k}\}_{k\in\mathbbm{N}}$ the arbitrary initial subsequence, and since $\mathcal{A}$ is totally ordered by Lemma \ref{lem:Contra} we may also assume either $\psi_{k}\searrow \psi$ or $\psi_{k}\nearrow \psi$ almost everywhere. In particular even if $u=P_{\psi_{\min}}$ we may suppose that $u_{k}$ converges weakly to a proper element $v\in\mathcal{E}^{1}(X,\omega,\psi)$ up to considering a further subsequence by definition of weak convergence to the point $P_{\psi_{\min}}$. In this case by abuse of notation we denote the function $v$, which depends on the subsequence chosen, by $u$. Note also that by Hartogs' Lemma we have $u_{k}\leq \psi_{k}+A,u\leq \psi+A$ for a uniform constant $A\in\mathbbm{R}_{\geq 0}$ since $|\sup_{X}u_{k}|\leq A$.\\
In the case $\psi_{k}\searrow \psi$, $v_{k}:=\big(\sup\{u_{j}\,:\, j\geq k\}\big)^{*}\in \mathcal{E}^{1}(X,\omega,\psi_{k})$ decreases to $u$. Thus $w_{k}:=P_{\omega}[\psi](v_{k})\in \mathcal{E}^{1}(X,\omega,\psi)$ decreases to $u$, which implies $d(u,w_{k})\to 0$ as $k\to \infty$ (if $u=P_{\psi_{\min}}$ we immediately have $w_{k}=P_{\psi_{\min}}$).\\
Moreover by Propositions \ref{prop:PropertiesE} and \ref{prop:PropProie} it follows that
\begin{multline*}
E_{\psi}(u)=\lim_{k\to \infty}E_{\psi}(w_{k})=AV_{\psi}-\lim_{k\to \infty}d(\psi+A,w_{k})\geq \lim_{k\to \infty}\big(AV_{\psi_{k}}-d(\psi_{k}+A,v_{k})\big)=\\
=\limsup_{k\to \infty} E_{\psi_{k}}(v_{k})\geq \lim_{k\to \infty}E_{\psi_{k}}(u_{k})=E_{\psi}(u)
\end{multline*}
since $\psi_{k}+A=P_{\omega}[\psi_{k}](A)$. Hence $\limsup_{k\to \infty}d(v_{k},u_{k})=\limsup_{k\to \infty}d(\psi_{k}+A,u_{k})-d(v_{k},\psi_{k}+A)=\lim_{k\to \infty} E_{\psi_{k}}(v_{k})-E_{\psi_{k}}(u_{k})=0$. Thus by the triangle inequality it is sufficient to show that $\limsup_{k\to \infty}d_{\mathcal{A}}(u,v_{k})=0$.\\
Next for any $C\in\mathbbm{R}$ we set $v_{k}^{C}:=\max(v_{k},\psi_{k}-C), u^{C}:=\max(u,\psi-C)$ and we observe that $d(\psi_{k}+A,v_{k}^{C})\to d(\psi+A,u^{C})$ by Proposition \ref{prop:AllNecessary} since $v_{k}^{C}\searrow u^{C}$. This implies that
\begin{multline*}
d(v_{k},v_{k}^{C})=d(\psi_{k}+A,v_{k})-d(\psi_{k}+A,v_{k}^{C})=AV_{\psi_{k}}-E_{\psi_{k}}(v_{k})-d(\psi_{k}+A,v_{k}^{C})\longrightarrow\\
\longrightarrow AV_{\psi}-E_{\psi}(u)-d(\psi+A,u^{C})= d(\psi+A,u)-d(\psi+A,u^{C})=d(u,u^{C}).
\end{multline*}
Thus, since $u^{C}\to u$ strongly, again by the triangle inequality it remains to estimate $d_{\mathcal{A}}(u,v_{k}^{C})$. Fix $\epsilon>0$ and $\phi_{\epsilon}\in \mathcal{P}_{\mathcal{H}_{\omega}}(X,\omega,\psi)$ such that $d(\phi_{\epsilon},u)\leq \epsilon $ (by Lemma \ref{lem:Density}). Then letting $\varphi\in\mathcal{H}_{\omega}$ such that $\phi_{\epsilon}=P_{\omega}[\psi](\varphi)$ and setting $\phi_{\epsilon,k}:=P_{\omega}[\psi_{k}](\varphi)$ by Proposition \ref{prop:AllNecessary} we have
$$
\limsup_{k\to \infty} d_{\mathcal{A}}(u,v_{k}^{C})\leq \limsup_{k\to \infty}\big(d(u,\phi_{\epsilon})+d_{\mathcal{A}}(\phi_{\epsilon},\phi_{\epsilon,k})+d(\phi_{\epsilon,k},v_{k}^{C})\big)\leq \epsilon+d(\phi_{\epsilon},u^{C})\leq 2\epsilon+d(u,u^{C}),
$$
which concludes the first case of $(ii)\Rightarrow (i)$ by the arbitrariety of $\epsilon$ since $u^{C}\to u$ strongly in $\mathcal{E}^{1}(X,\omega,\psi)$.\\
Next assume that $\psi_{k}\nearrow \psi$ almost everywhere. In this case we clearly may assume $V_{\psi_{k}}>0$ for any $k\in\mathbbm{N}$. Then $v_{k}:=\big(\sup\{u_{j}\, :\, j\geq k\}\big)^{*}\in\mathcal{E}^{1}(X,\omega,\psi)$ decreases to $u$. Moreover setting $w_{k}:=P_{\omega}[\psi_{k}](v_{k})\in\mathcal{E}^{1}(X,\omega,\psi_{k})$ and combining the monotonicity of $E_{\psi_{k}}(\cdot)$, the upper semicontinuity of $E_{\cdot}(\cdot)$ (Proposition \ref{prop:USC}) and the contraction property of Proposition \ref{prop:PropProie} we obtain
\begin{multline*}
E_{\psi}(u)=\lim_{k\to \infty}E_{\psi}(v_{k})=AV_{\psi}-\lim_{k\to \infty}d(v_{k},\psi+A)\leq\liminf_{k\to \infty}\big(AV_{\psi_{k}}-d(w_{k},\psi_{k}+A)\big)=\\
=\liminf_{k\to \infty}E_{\psi_{k}}(w_{k})\leq \limsup_{k\to \infty} E_{\psi_{k}}(w_{k})\leq E_{\psi}(u),
\end{multline*}
i.e. $E_{\psi_{k}}(w_{k})\to E_{\psi}(u)$ as $k\to \infty$. As a easy consequence we also get $d(w_{k},u_{k})=E_{\psi_{k}}(w_{k})-E_{\psi_{k}}(u_{k})\to 0$, thus it is sufficient to prove that
$$
\limsup_{k\to \infty}d_{\mathcal{A}}(u,w_{k})=0.
$$
Similarly to the previous case, fix $\epsilon>0$ and let $\phi_{\epsilon}=P_{\omega}[\psi](\varphi_{\epsilon})$ for $\varphi\in\mathcal{H}_{\omega}$ such that $d(u,\phi_{\epsilon})\leq \epsilon$. Again Proposition \ref{prop:PropProie} and Proposition \ref{prop:AllNecessary} yield
\begin{multline*}
\limsup_{k\to \infty}d_{\mathcal{A}}(u,w_{k})\leq \epsilon+\limsup_{k\to \infty}\big( d_{\mathcal{A}}\big(\phi_{\epsilon},P_{\omega}[\psi_{k}](\phi_{\epsilon})\big)+d\big(P_{\omega}[\psi_{k}](\phi_{\epsilon}),w_{k}\big)\big)\leq\\
\leq \epsilon+ \limsup_{k\to \infty} \big( d_{\mathcal{A}}\big(\phi_{\epsilon},P_{\omega}[\psi_{k}](\phi_{\epsilon})\big)+d(\phi_{\epsilon},v_{k})\big)\leq 2\epsilon,
\end{multline*}
which concludes the first part.\\
\textbf{Implication $\mathbf{(i)\Rightarrow (ii)}$ if $\mathbf{u\neq P_{\psi_{\min}}}$ while $\mathbf{(i)}$ implies the dichotomy if $\mathbf{u=P_{\psi_{\min}}}$.\\}
If $u\neq P_{\psi_{\min}}$, Proposition \ref{prop:StrongImpliesWeak} implies that $u_{k}\to u$ weakly and in particular that $|\sup_{X}u_{k}|\leq A$. Thus it remains to prove that $E_{\psi_{k}}(u_{k})\to E_{\psi}(u)$.\\
If $u=P_{\psi_{\min}}$ then again by Proposition \ref{prop:StrongImpliesWeak} it remains to show that $E_{\psi_{k}}(u_{k})\to 0$ assuming $u_{k_{h}}\to P_{\psi_{\min}}$ strongly and weakly. Note that we also have $|\sup_{X}u_{k}|\leq A$ for a uniform constant $A\in\mathbbm{R}$ by definition of weak convergence to $P_{\psi_{\min}}$.\\
So, since by an easy contradiction argument it is enough to prove that any subsequence of $\{u_{k}\}_{k\in\mathbbm{N}}$ admits a further subsequence such that the convergence of the energies holds, without loss of generality we may assume that $u_{k}\to u\in\mathcal{E}^{1}(X,\omega,\psi)$ weakly even in the case $V_{\psi}=0$ (i.e. when, with abuse of notation, $u=P_{\psi_{\min}}$).\\
Therefore we want to show the existence of a further subsequence $\{u_{k_{h}}\}_{h\in\mathbbm{N}}$ such that $E_{\psi_{k_{h}}}(u_{k_{h}})\to E_{\psi}(u)$ (note that if $V_{\psi}=0$ then $E_{\psi}(u)=0$). It easily follows that
$$
|E_{\psi_{k}}(u_{k})-E_{\psi}(u)|\leq|d(\psi_{k}+A,u_{k})-d(\psi+A,u)|+A|V_{\psi_{k}}-V_{\psi}|\leq d_{\mathcal{A}}(u,u_{k})+d(\psi_{k}+A,\psi+A)+A|V_{\psi_{k}}-V_{\psi}|,
$$
and this leads to $\lim_{k\to \infty}E_{\psi_{k}}(u_{k})=E_{\psi}(u)$ by Proposition \ref{prop:AllNecessary} since $\psi_{k}+A=P_{\omega}[\psi_{k}](A)$ and $\psi+A=P_{\omega}[\psi](A)$. Hence $E_{\psi_{k}}(u_{k})\to E_{\psi}(u)$ as requested.
\end{proof}
Note that in Theorem \ref{thm:OldPropA} the case $(b)$ may happen (Remark \ref{rem:ImpRem}) but obviously one can consider
$$
X_{\mathcal{A},norm}=\bigsqcup_{\psi\in\overline{\mathcal{A}}}\mathcal{E}^{1}_{norm}(X,\omega,\psi)
$$
to exclude such pathology.\\
The strong convergence also implies the convergence in $\psi'$-capacity for any $\psi'\in\mathcal{M}^{+}$ as our next result shows.
\begin{thm}
\label{thm:OldPropB}
Let $\psi_{k},\psi\in\mathcal{A}$, and let $u_{k}\in\mathcal{E}^{1}(X,\omega,\psi_{k})$ strongly converging to $u\in\mathcal{E}^{1}(X,\omega,\psi)$. Assuming also that $V_{\psi}>0$. Then there exists a subsequence $\{u_{k_{j}}\}_{j\in\mathbbm{N}}$ such that the sequences $w_{j}:=\big(\sup\{u_{k_{s}}\, : \, s\geq j\}\big)^{*}$, $v_{j}:=P_{\omega}(u_{k_{j}},u_{k_{j+1}},\dots)$ belong to $X_{\mathcal{A}}$, satisfy $v_{j}\leq u_{k_{j}}\leq w_{j}$ and converge strongly and monotonically to $u$. In particular $u_{k}\to u$ in $\psi'$-capacity for any $\psi'\in\mathcal{M}^{+}$ and $MA_{\omega}(u_{k}^{j},\psi_{k}^{n-j})\to MA_{\omega}(u^{k},\psi^{n-j})$ weakly for any $j\in\{0,\dots,n\}$.
\end{thm}
\begin{proof}
We first observe that by Theorem \ref{thm:OldPropA} $u_{k}\to u$ weakly and $E_{\psi_{k}}(u_{k})\to E_{\psi}(u)$. In particular $\sup_{X}u_{k}$ is uniformly bounded and the sequence of $\omega$-psh $w_{k}:=\big(\sup\{u_{j}\, :\, j\geq k\}\big)^{*}$ decreases to $u$. \\
Up to considering a subsequence we may assume either $\psi_{k}\searrow \psi$ or $\psi_{k}\nearrow \psi$ almost everywhere. We treat the two cases separately.\\
Assume first that $\psi_{k}\searrow\psi$. Since clearly $w_{k}\in\mathcal{E}^{1}(X,\omega,\psi_{k})$ and $E_{\psi_{k}}(w_{k})\geq E_{\psi_{k}}(u_{k})$, Theorem \ref{thm:OldPropA} and Proposition \ref{prop:USC} yields 
$$
E_{\psi}(u)=\lim_{k\to \infty}E_{\psi_{k}}(u_{k})\leq \limsup_{k\to \infty}E_{\psi_{k}}(w_{k})\leq E_{\psi}(u),
$$
i.e. $w_{k}\to u$ strongly. Thus up to considering a further subsequence we can suppose that $d(u_{k},w_{k})\leq 1/2^{k}$ for any $k\in\mathbbm{N}$.\\
Next similarly as during the proof of Proposition \ref{prop:CapacityPsi} we define $v_{j,l}:=P_{\omega}(u_{j},\dots,u_{j+l})$ for any $j,l\in\mathbbm{N}$, observing that $v_{j,l}\in\mathcal{E}^{1}(X,\omega,\psi_{j+l})$. Thus the function $v_{j,l}^{u}:=P_{\omega}(u,v_{j,l})\in\mathcal{E}^{1}(X,\omega,\psi)$ satisfies
\begin{multline}
\label{eqn:Uf}
d(u,v_{j,l}^{u})\leq\int_{X}(u-v_{j,l}^{u})MA_{\omega}(v_{j,l}^{u})\leq \int_{\{v_{j,l}^{u}=v_{j,l}\}}(u-v_{j,l})MA_{\omega}(v_{j,l})\leq \\
\leq\sum_{s=j}^{j+l}\int_{X}(w_{s}-u_{s})MA_{\omega}(u_{s})\leq (n+1)\sum_{s=j}^{j+l}d(w_{s},u_{s})\leq \frac{(n+1)}{2^{j-1}},
\end{multline}
where we combined Proposition \ref{prop:PropertiesE} and Lemma $3.7$ in \cite{DDNL17b}. Therefore by Proposition \ref{prop:CompactL1} $v_{j,l}^{u}$ converges decreasingly and strongly in $\mathcal{E}^{1}(X,\omega,\psi)$ to a function $\phi_{j}$ which satisfies $\phi_{j}\leq u$. \\
Similarly $\int_{\{P_{\omega}(u,v_{j,l}^{u})=u\}}(v_{j,l}^{u}-u)MA_{\omega}(u)\leq \int_{X}|v_{j,1}^{u}-u|MA_{\omega}(u)<\infty$ by Corollary \ref{cor:ForL1}, which implies that $v_{j,l}$ converges decreasingly to $v_{j}\in\mathcal{E}^{1}(X,\omega,\psi)$ such that $u\geq v_{j}\geq \phi_{j}$ since $v_{j}\leq u_{s}$ for any $s\geq j$ and $v_{j,l}\geq v_{j,l}^{u}$. Hence from (\ref{eqn:Uf}) we obtain
$$
d(u,v_{j})\leq d(u,\phi_{j})=\lim_{l\to \infty}d(u,v_{j,l}^{u})\leq \frac{(n+1)}{2^{j-1}},
$$
i.e. $v_{j}$ converges increasingly and strongly to $u$ as $j\to \infty$.\\
Next assume $\psi_{k}\nearrow \psi$ almost everywhere. In this case $w_{k}\in\mathcal{E}^{1}(X,\omega,\psi)$ for any $k\in\mathbbm{N}$, and clearly $w_{k}$ converges strongly and decreasingly to $u$. On the other hand, letting $w_{k,k}:=P_{\omega}[\psi_{k}](w_{k})$ we observe that $w_{k,k}\to u$ weakly since $w_{k}\geq w_{k,k}\geq u_{k}$ and
$$
E_{\psi}(u)=\lim_{k\to \infty}E_{\psi_{k}}(u_{k})\leq \limsup_{k\to \infty} E_{\psi_{k}}(w_{k,k})\leq E_{\psi}(u)
$$
by Theorem \ref{thm:OldPropA} and Proposition \ref{prop:USC}, i.e. $w_{k,k}\to u$ strongly again by Theorem \ref{thm:OldPropA}. Thus, similarly to the previous case, we may assume that $d(u_{k},w_{k,k})\leq 1/2^{k}$ up to considering a further subsequence. Therefore setting $v_{j,l}:=P_{\omega}(u_{j},\dots,u_{j+l})\in\mathcal{E}^{1}(X,\omega,\psi_{j})$, $u^{j}:=P_{\omega}[\psi_{j}](u)$ and $v_{j,l}^{u^{j}}:=P_{\omega}\big(v_{j,l},u^{j}\big)$ we obtain
\begin{gather}
\label{eqn:Uf2}
d\big(u^{j},v_{j,l}^{u^{j}}\big)\leq \int_{X}\big(u^{j}-v_{j,l}^{u^{j}}\big)MA_{\omega}(v_{j,l}^{u^{j}})\leq \sum_{s=j}^{j+l}\int_{X}(w_{s,s}-u_{s})MA_{\omega}(u_{s})\leq \frac{(n+1)}{2^{j-1}}
\end{gather}
proceeding similarly as before. This implies that $v_{j,l}^{u^{j}}$ and $v_{j,l}$ converge decreasingly and strongly respectively to functions $\phi_{j},v_{j}\in\mathcal{E}^{1}(X,\omega,\psi_{j})$ as $l\to +\infty$ which satisfy $\phi_{j}\leq v_{j}\leq u^{j}$. Therefore combining (\ref{eqn:Uf2}), Proposition \ref{prop:AllNecessary} and the triangle inequality we get
$$
\limsup_{j\to \infty}d_{\mathcal{A}}(u,v_{j})\leq \limsup_{j\to \infty}\Big(d_{\mathcal{A}}(u,u^{j})+d(u^{j},\phi_{j})\Big)\leq \limsup_{j\to \infty}\Big(d_{\mathcal{A}}(u,u^{j})+\frac{(n+1)}{2^{j-1}}\Big)=0.
$$
Hence $v_{j}$ converges strongly and increasingly to $u$, so $v_{j}\nearrow u$ almost everywhere (Propositon \ref{prop:StrongImpliesWeak}) and the first part of the proof is concluded.\\
The convergence in $\psi'$-capacity and the weak convergence of the mixed Monge-Ampère measures follow exactly as seen during the proof of Proposition \ref{prop:CapacityPsi}.
\end{proof}
We observe that the assumption $u\neq P_{\psi_{\min}}$ if $V_{\psi_{\min}}=0$ in Theorem \ref{thm:OldPropB} is obviously necessary as the counterexample of Remark \ref{rem:ImpRem} shows. On the other hand if $d_{\mathcal{A}}(u_{k},P_{\psi_{\min}})\to 0$ then trivially $MA_{\omega}(u_{k}^{j},\psi_{k}^{n-j})\to 0$ weakly as $k\to \infty$ for any $j\in\{0,\dots,n\}$ as a consequence of $V_{\psi_{k}}\searrow 0$.
\subsection{Proof of Theorem \ref{thmB}}
\begin{defn}
We define $Y_{\mathcal{A}}$ as
$$
Y_{\mathcal{A}}:=\bigsqcup_{\psi\in\overline{\mathcal{A}}}\mathcal{M}^{1}(X,\omega,\psi),
$$
and we endow it with its natural \emph{strong topology} given as the coarsest refinement of the weak topology such that $E^{*}_{\cdot}$ becomes continuous, i.e. $V_{\psi_{k}}\mu_{k}$ converges strongly to $V_{\psi}\mu$ if and only if $V_{\psi_{k}}\mu_{k}\to V_{\psi}\mu$ weakly and $E^{*}_{\psi_{k}}(\mu_{k})\to E^{*}_{\psi}(\mu)$ as $k\to \infty$.  
\end{defn}
Observe that $Y_{\mathcal{A}}\subset \{\mbox{non-pluripolar measures of total mass belonging to}\, [V_{\psi_{\min}}, V_{\psi_{\max}}]\}$ where clearly $\psi_{\max}:=\sup\mathcal{A}$. As stated in the Introduction, the denomination is coherent with \cite{BBEGZ16} since if $\psi=0\in \overline{\mathcal{A}}$ then the induced topology on $\mathcal{M}^{1}(X,\omega)$ coincides with the strong topology as defined in \cite{BBEGZ16}.\\
We also recall that
$$
X_{\mathcal{A},norm}:=\bigsqcup_{\psi\in\overline{\mathcal{A}}}\mathcal{E}^{1}_{norm}(X,\omega,\psi)
$$
where $\mathcal{E}^{1}_{norm}(X,\omega,\psi):=\{u\in\mathcal{E}^{1}(X,\omega,\psi)\,\, \mbox{such that}\, \sup_{X}u=0\}$ (if $V_{\psi_{\min}}=0$ then we clearly assume $P_{\psi_{\min}}\in X_{\mathcal{A},norm}$).
\begin{reptheorem}{thmB}
The Monge-Ampère map
$$
MA_{\omega}:(X_{\mathcal{A},norm},d_{\mathcal{A}})\to (Y_{\mathcal{A}},strong)
$$
is a homeomorphism.
\end{reptheorem}
\begin{proof}
The map is a bijection as a consequence of Lemma \ref{lem:HomeoV} and Proposition \ref{prop:OldPropC} defining clearly $MA_{\omega}(P_{\psi_{\min}}):=0$, i.e. to be the null measure.\\
\textbf{Step $1$: Continuity.} Assume first that $V_{\psi_{\min}}=0$ and that $d_{\mathcal{A}}(u_{k},P_{\psi_{\min}})\to 0$ as $k\to \infty$. Then easily $MA_{\omega}(u_{k})\to 0$ weakly. Moreover, assuming $u_{k}\neq P_{\psi_{\min}}$ for any $k$, it follows from Proposition \ref{prop:PropertiesE} that
\begin{gather*}
E_{\psi_{k}}^{*}\big(MA_{\omega}(u_{k})/V_{\psi_{k}}\big)=E_{\psi_{k}}(u_{k})+\int_{X}(\psi_{k}-u_{k})MA_{\omega}(u_{k})\leq\\
\leq \frac{n}{n+1}\int_{X}(\psi_{k}-u_{k})MA_{\omega}(u_{k})\leq -n E_{\psi_{k}}(u_{k})\to 0
\end{gather*}
as $k\to \infty$ where the convergence is given by Theorem \ref{thm:OldPropA}. Hence $MA_{\omega}(u_{k})\to 0$ strongly in $Y_{\mathcal{A}}$.\\
We can now assume that $u\not=P_{\psi_{\min}}$.\\
Theorem \ref{thm:OldPropB} immediately gives the weak convergence of $MA_{\omega}(u_{k})$ to $MA_{\omega}(u)$. Fix $\varphi_{j}\in \mathcal{H}_{\omega}$ be a decreasing sequence converging to $u$ such that $d\big(u,P_{\omega}[\psi](\varphi_{j})\big)\leq 1/j$ for any $j\in\mathbbm{N}$ (\cite{BK07}) and set $v_{k,j}:=P_{\omega}[\psi_{k}](\varphi_{j})$ and $v_{j}:=P_{\omega}[\psi](\varphi_{j})$. Observe also that as a consequence of Proposition \ref{prop:AllNecessary} and Theorem \ref{thm:OldPropA}, for any $j\in\mathbbm{N}$ there exists $k_{j}\gg 0$ big enough such that $d(\psi_{k},v_{k,j})\leq d_{\mathcal{A}}(\psi_{k},\psi)+d(\psi,v_{j})+d_{\mathcal{A}}(v_{j},v_{k,j})\leq d(\psi,v_{j})+1\leq C$ for any $k\geq k_{j}$, where $C$ is a uniform constant independent on $j\in\mathbbm{N}$. Therefore combining again Theorem \ref{thm:OldPropA} with Lemma \ref{lem:Equicontinuity} and Proposition \ref{prop:LastLast} we obtain
\begin{multline}
\label{eqn:Need}
\limsup_{k\to \infty}\Big|E^{*}_{\psi_{k}}\big(MA_{\omega}(u_{k})/V_{\psi_{k}}\big)-E^{*}_{\psi_{k}}\big(MA_{\omega}(v_{k,j})/V_{\psi_{k}}\big)\Big|\leq\\
\leq \limsup_{k\to \infty}\Big(\big|E_{\psi_{k}}(u_{k})-E_{\psi_{k}}(v_{k,j})\big|+\Big|\int_{X}(\psi_{k}-u_{k})\big(MA_{\omega}(u_{k})-MA_{\omega}(v_{k,j})\big)\Big|+\Big|\int_{X}(v_{k,j}-u_{k})MA_{\omega}(v_{k,j})\Big|\Big)\leq\\
\leq \big|E_{\psi}(u)-E_{\psi}(v_{j})\big|+\limsup_{k\to \infty}C I_{\psi_{k}}(u_{k},v_{k,j})^{1/2}+ \int_{X}(v_{j}-u)MA_{\omega}(v_{j})
\end{multline}
since clearly we may assume that either $\psi_{k}\searrow \psi$ or $\psi_{k}\nearrow \psi$ almost everywhere, up to considering a subsequence. On the other hand, if $k\geq k_{j}$, Proposition \ref{prop:ForL1} implies $I_{\psi_{k}}(u_{k},v_{k,j})\leq 2f_{\tilde{C}}\big(d(u_{k},v_{k,j})\big)$ where $\tilde{C}$ is a uniform constant independent of $j,k$ and $f_{\tilde{C}}:\mathbbm{R}_{\geq 0}\to \mathbbm{R}_{\geq 0}$ is a continuous increasing function such that $f_{\tilde{C}}(0)=0$. Hence continuing the estimates in (\ref{eqn:Need}) we get
\begin{equation}
\label{eqn:Need2}
(\ref{eqn:Need})\leq \big|E_{\psi}(u)-E_{\psi}(v_{j})\big|+2Cf_{\tilde{C}}\big(d(u,v_{j})\big)+d(v_{j},u)
\end{equation}
using also Propositions \ref{prop:PropertiesE} and \ref{prop:AllNecessary}. Letting $j\to \infty$ in (\ref{eqn:Need2}), it follows that
$$
\limsup_{j\to \infty}\limsup_{k\to \infty}\Big|E^{*}_{\psi_{k}}\big(MA_{\omega}(u_{k})/V_{\psi_{k}}\big)-E^{*}_{\psi_{k}}\big(MA_{\omega}(v_{k,j})/V_{\psi_{k}}\big)\Big|=0
$$
since $v_{j}\searrow u$. Furthermore it is easy to check that $E^{*}_{\psi_{k}}\big(MA_{\omega}(v_{k,j})/V_{\psi_{k}}\big)\to E^{*}_{\psi}\big(MA_{\omega}(v_{j})/V_{\psi}\big)$ as $k\to \infty$ for $j$ fixed by Lemma \ref{lem:Equicontinuity} and Proposition \ref{prop:AllNecessary}. Therefore the convergence
\begin{equation}
\label{eqn:Need3}
E_{\psi}^{*}\big(MA_{\omega}(v_{j})/V_{\psi}\big)\to E^{*}_{\psi}\big(MA_{\omega}(u)/V_{\psi}\big)
\end{equation}
as $j\to \infty$ given by Theorem \ref{thmA} concludes this step. \\
\textbf{Step $2$: Continuity of the inverse.} Assume $u_{k}\in\mathcal{E}^{1}_{norm}(X,\omega,\psi_{k}), u\in\mathcal{E}^{1}_{norm}(X,\omega,\psi)$ such that $MA_{\omega}(u_{k})\to MA_{\omega}(u)$ strongly. Note that when $\psi=\psi_{\min}$ and $V_{\psi_{\min}}=0$ the assumption does not depend on the function $u$ chosen. Clearly this implies $V_{\psi_{k}}\to V_{\psi}$ which leads to $\psi_{k}\to \psi$ as $k\to \infty$ by Lemma \ref{lem:HomeoV} since $\mathcal{A}\subset \mathcal{M}^{+}$ is totally ordered. Hence, up to considering a subsequence, we may assume that $\psi_{k}\to \psi$ monotonically almost everywhere. We keep the same notations of the previous step for $v_{k,j},v_{j}$. We may also suppose that $V_{\psi_{k}}>0$ for any $k\in\mathbbm{N}$ big enough otherwise it would be trivial.\\
The strategy is to proceed similarly as during the proof of Theorem \ref{thmA}, i.e. we want first to prove that $I_{\psi_{k}}(u_{k},v_{k,j})\to 0$ as $k,j\to \infty$ in this order. Then we want to use this to prove that the unique weak accumulation point of $\{u_{k}\}_{k\in\mathbbm{N}}$ is $u$. Finally we will deduce also the convergence of the $\psi_{k}$-relative energies to conclude that $u_{k}\to u$ strongly thanks to Theorem \ref{thm:OldPropA}.\\
By Lemma \ref{lem:Related}
\begin{multline}
\label{eqn:Suuu}
(n+1)^{-1}I_{\psi_{k}}(u_{k},v_{k,j})\leq E_{\psi_{k}}(u_{k})-E_{\psi_{k}}(v_{k,j})+\int_{X}(v_{k,j}-u_{k})MA_{\omega}(u_{k})=\\
=E^{*}_{\psi_{k}}\big(MA_{\omega}(u_{k})/V_{\psi_{k}}\big)-E^{*}_{\psi_{k}}\big(MA_{\omega}(v_{k,j})/V_{\psi_{k}}\big)+\int_{X}(v_{k,j}-\psi_{k})\big(MA_{\omega}(u_{k})-MA_{\omega}(v_{k,j})\big)
\end{multline}
for any $j,k$. Moreover by Step $1$ and Proposition \ref{prop:AllNecessary} we know that $E^{*}_{\psi_{k}}\big(MA_{\omega}(v_{k,j})/V_{\psi_{k}}\big)$ converges, as $k\to +\infty$, respectively to $0$ if $V_{\psi}=0$ and to $E^{*}_{\psi}\big(MA_{\omega}(v_{j})/V_{\psi}\big)$ if $V_{\psi}>0$. Next by Lemma \ref{lem:Equicontinuity}
$$
\int_{X}(v_{k,j}-\psi_{k})MA_{\omega}(v_{k,j})\to \int_{X}(v_{j}-\psi)MA_{\omega}(v_{j})
$$
letting $k\to \infty$. So if $V_{\psi}=0$ then from $\lim_{k\to \infty}\sup_{X}(v_{k,j}-\psi_{k})=\sup_{X}(v_{j}-\psi)=\sup_{X}v_{j}$ we easily get $\limsup_{k\to \infty} I_{\psi_{k}}(u_{k},v_{k,j})=0$. Thus we may assume $V_{\psi}>0$ and it remains to estimate $\int_{X}(v_{k,j}-\psi_{k})MA_{\omega}(u_{k})$ from above.\\
We set $f_{k,j}:=v_{k,j}-\psi_{k}$ and analogously to the proof of Theorem \ref{thmA} we construct a sequence of smooth functions $f_{j}^{s}:=v_{j}^{s}-\psi^{s}$ converging in capacity to $f_{j}:=v_{j}-\psi$ and satisfying $||f_{j}^{s}||_{L^{\infty}}\leq 2||f_{j}||_{L^{\infty}}\leq 2||\varphi_{j}||_{L^{\infty}}$. Here $v_{j}^{s},\psi^{s}$ are sequences of $\omega$-psh functions decreasing respectively to $v_{j}, \psi$. Then we write
\begin{equation}
\label{eqn:Uffy1}
\int_{X}f_{k,j}MA_{\omega}(u_{k})=\int_{X}(f_{k,j}-f_{j}^{s})MA_{\omega}(u_{k})+\int_{X}f_{j}^{s}MA_{\omega}(u_{k})
\end{equation}
and we observe that $\limsup_{s\to \infty}\limsup_{k\to \infty}\int_{X}f_{j}^{s}MA_{\omega}(u_{k})=\int_{X}f_{j}MA_{\omega}(u)$ since $MA_{\omega}(u_{k})\to MA_{\omega}(u)$ weakly, $f_{j}^{s}\in C^{\infty}$, $f_{j}^{s}$ converges to $f_{j}$ in capacity and $||f_{j}^{s}||_{L^{\infty}}\leq 2||f_{j}||_{L^{\infty}}$. While we claim that the first term on the right-hand side of (\ref{eqn:Uffy1}) goes to $0$ letting $k,s\to \infty$ in this order. Indeed for any $\delta>0$
\begin{multline}
\label{eqn:UffyC1}
\int_{X}(f_{k,j}-f_{j})MA_{\omega}(u_{k})\leq \delta V_{\psi_{k}}+2||\varphi_{j}||_{L^{\infty}}\int_{\{f_{k,j}-f_{j}>\delta\}}MA_{\omega}(u_{k})\leq \delta V_{\psi_{k}}+2||\varphi_{j}||_{L^{\infty}}\int_{\{|h_{k,j}-h_{j}|>\delta\}}MA_{\omega}(u_{k})
\end{multline}
where we set $h_{k,j}:=v_{k,j}, h_{j}:=v_{j}$ if $\psi_{k}\searrow \psi$ and $h_{k,j}:=\psi_{k}, h_{j}:=\psi$ if instead $\psi_{k}\nearrow \psi$ almost everywhere. Moreover since $\{|h_{k,j}-h_{j}|>\delta\}\subset \{|h_{l,j}-h_{j}|>\delta\}$ for any $l\leq k$, from (\ref{eqn:UffyC1}) we obtain
\begin{multline*}
\limsup_{k\to \infty}\int_{X}(f_{k,j}-f_{j})MA_{\omega}(u_{k})\leq\delta V_{\psi}+\limsup_{l\to \infty}\limsup_{k\to \infty} 2||\varphi_{j}||_{L^{\infty}}\int_{\{|h_{l,j}-h_{j}|\geq\delta\}} MA_{\omega}(u_{k})\leq \\
\leq \delta V_{\psi} + \limsup_{l\to \infty} 2||\varphi_{j}||_{L^{\infty}}\int_{\{|h_{l,j}-h_{j}|\geq \delta\}}MA_{\omega}(u)=\delta V_{\psi}
\end{multline*}
where we also used that $\{|h_{l,j}-h_{j}|\geq \delta\}$ is a closed set in the plurifine topology since it is equal to $\{v_{l,j}-v_{j}\geq \delta\}$ if $\psi_{l}\searrow \psi$ and to $\{\psi-\psi_{l}\geq \delta\}$ if $\psi_{l}\nearrow \psi$ almost everywhere. Hence $\limsup_{k\to \infty}\int_{X}(f_{k,j}-f_{j})MA_{\omega}(u_{k})\leq 0$. Similarly we also get $\limsup_{s\to \infty}\limsup_{k\to \infty}\int_{X}(f_{j}-f_{j}^{s})MA_{\omega}(u_{k})\leq 0$. (see also the proof of Theorem \ref{thmA}).\\
Summarizing from (\ref{eqn:Suuu}), we obtain
\begin{multline}
\label{eqn:Uffy2}
\limsup_{k\to\infty}(n+1)^{-1}I_{\psi_{k}}(u_{k},v_{k,j})\leq E^{*}_{\psi}\big(MA_{\omega}(u)/V_{\psi}\big)-E_{\psi}^{*}\big(MA_{\omega}(v_{j})/V_{\psi}\big)+\\
+\int_{X}(v_{j}-\psi)MA_{\omega}(u)-\int_{X}(v_{j}-\psi)MA_{\omega}(v_{j})=:F_{j},
\end{multline}
and $F_{j}\to 0$ as $j\to \infty$ by Step $1$ and Proposition \ref{prop:LastLast} since $\mathcal{E}^{1}(X,\omega,\psi)\ni v_{j}\searrow u\in\mathcal{E}^{1}_{norm}(X,\omega,\psi)$, hence strongly.\\ 
Next  by Lemma \ref{lem:Related} $u_{k}\in X_{\mathcal{A},C}$ for $C\gg 1$ since $E^{*}\big(MA_{\omega}(u_{k})/V_{\psi_{k}}\big)=J_{u_{k}}^{\psi}(\psi)$ and $\sup_{X}u_{k}=0$, thus up to considering a further subsequence $u_{k}\to w\in \mathcal{E}^{1}_{norm}(X,\omega,\psi)$ weakly where $d(w,\psi)\leq C$. Indeed if $V_{\psi}>0$ this follows from Proposition \ref{prop:CompactL1} while it is trivial if $V_{\psi}=0$. In particular by Lemma \ref{lem:Equicontinuity}
\begin{gather}
\label{eqn:Su2}
\int_{X}(\psi_{k}-u_{k})MA_{\omega}(v_{k,j})\to \int_{X}(\psi-w)MA_{\omega}(v_{j})\\
\label{eqn:Su3}
\int_{X}(v_{k,j}-u_{k})MA_{\omega}(v_{k,j})\to \int_{X}(v_{j}-w)MA_{\omega}(v_{j})
\end{gather}
as $j\to \infty$. Therefore if $V_{\psi}=0$ then combining $I_{\psi_{k}}(u_{k},v_{k,j})\to 0$ as $k\to \infty$ with (\ref{eqn:Su3}) and Lemma \ref{lem:Related}, we obtain
\begin{gather*}
\limsup_{k\to \infty}\Big(-E_{\psi_{k}}(u_{k})+E_{\psi_{k}}(v_{k,j})\Big)\leq \limsup_{k\to \infty}\Big(\frac{n}{n+1}I_{\psi_{k}}(u_{k},v_{k,j})+\Big|\int_{X}(v_{k,j}-u_{k})MA_{\omega}(v_{k,j})\Big|\Big)=0.
\end{gather*}
This implies that $d(\psi_{k},u_{k})=-E_{\psi_{k}}(u_{k})\to 0$ as $k\to \infty$, i.e. that $d_{\mathcal{A}}(P_{\psi_{\min}},u_{k})\to 0$ using Theorem \ref{thm:OldPropA}. Thus we may assume from now until the end of the proof that $V_{\psi}>0$. \\
By (\ref{eqn:Su2}) and Proposition \ref{prop:USC} it follows that
\begin{multline}
\label{eqn:SuFin1}
\limsup_{k\to \infty}\Big(E^{*}_{\psi_{k}}\big(MA_{\omega}(u_{k})/V_{\psi_{k}}\big)+\int_{X}(\psi_{k}-u_{k})\big(MA_{\omega}(v_{k,j})-MA_{\omega}(u_{k})\big)\Big)=\\
=\limsup_{k\to\infty}\Big(E_{\psi_{k}}(u_{k})+\int_{X}(\psi_{k}-u_{k})MA_{\omega}(v_{k,j})\Big)\leq E_{\psi}(w)+\int_{X}(\psi-w)MA_{\omega}(v_{j}).
\end{multline}
On the other hand by Proposition \ref{prop:LastLast} and (\ref{eqn:Uffy2}),
\begin{equation}
\label{eqn:SuFin2}
\limsup_{k\to \infty}\Big|\int_{X}(\psi_{k}-u_{k})\big(MA_{\omega}(v_{k,j})-MA_{\omega}(u_{k})\big)\Big|\leq C F_{j}^{1/2}.
\end{equation}
In conclusion by the triangle inequality combining (\ref{eqn:SuFin1}) and (\ref{eqn:SuFin2}) we get
\begin{multline*}
E_{\psi}(u)+\int_{X}(\psi-u)MA_{\omega}(u)=\lim_{k\to\infty} E^{*}\big(MA_{\omega}/(u_{k})/V_{\psi_{k}}\big)\leq\\
\leq \limsup_{j\to \infty}\Big(E_{\psi}(w)+\int_{X}(\psi-w)MA_{\omega}(v_{j})+CF_{j}^{1/2}\Big)=E_{\omega}(w)+\int_{X}(\psi-w)MA_{\omega}(u)
\end{multline*}
since $F_{j}\to 0$, i.e. $w\in\mathcal{E}^{1}_{norm}(X,\omega,\psi)$ is a maximizer of $F_{MA_{\omega}(u)/V_{\psi},\psi}$. Hence $w=u$ (Proposition \ref{prop:OldPropC}), i.e. $u_{k}\to u$ weakly. Furthermore, similarly to the case $V_{\psi}=0$, Lemma \ref{lem:Related} and (\ref{eqn:Su3}) imply
\begin{gather*}
E_{\psi}(v_{j})-\liminf_{k\to \infty}E_{\psi_{k}}(u_{k})=\limsup_{k\to \infty}\Big(-E_{\psi_{k}}(u_{k})+E_{\psi_{k}}(v_{k,j})\Big)\leq\\
\leq \limsup_{k\to \infty}\Big(\frac{n}{n+1}I_{\psi_{k}}(u_{k},v_{k,j})+\Big|\int_{X}(u_{k}-v_{j,k})MA_{\omega}(v_{k,j})\Big|\Big)\leq \frac{n}{n+1}F_{j}+ \Big|\int_{X}(u-v_{j})MA_{\omega}(v_{j})\Big|
\end{gather*}
Finally letting $j\to \infty$, since $v_{j}\to u$ strongly, we obtain $\liminf_{k\to \infty}E_{\psi_{k}}(u_{k})\geq \lim_{j\to \infty}E_{\psi}(v_{j})=E_{\psi}(u)$. Hence $E_{\psi_{k}}(u_{k})\to E_{\psi}(u)$ by Proposition \ref{prop:USC} which implies $d_{\mathcal{A}}(u_{k},u)\to 0$ by Theorem \ref{thm:OldPropA} and concludes the proof.
\end{proof}
\section{Stability of Complex Monge-Ampère equations.}
\label{sec:CMAE}
As stated in the Introduction, we want to use the homeomorphism of Theorem \ref{thmB} to deduce the strong stability of solutions of complex Monge-Ampère equations with prescribed singularities when the measures have uniformly bounded $L^{p}$ density for $p>1$.
\begin{reptheorem}{thmC}
Let $\mathcal{A}:=\{\psi_{k}\}_{k\in\mathbbm{N}}\subset \mathcal{M}^{+}$ be totally ordered, and let $\{f_{k}\}_{k\in\mathbbm{N}}\subset L^{1}$ a sequence of non-negative functions such that $f_{k}\to f\in L^{1}\setminus \{0\}$ and such that $\int_{X}f_{k}\omega^{n}=V_{\psi_{k}}$ for any $k\in\mathbbm{N}$. Assume also that there exists $p>1$ such that $||f_{k}||_{L^{p}},||f||_{L^{p}}$ are uniformly bounded. Then $\psi_{k}\to \psi\in \overline{\mathcal{A}}\subset \mathcal{M}^{+}$, and the sequence of solutions of
\begin{equation}
\label{eqn:K-case}
\begin{cases}
MA_{\omega}(u_{k})=f_{k}\omega^{n}\\
u_{k}\in\mathcal{E}^{1}_{norm}(X,\omega,\psi_{k})
\end{cases}
\end{equation}
converges strongly to $u\in X_{\mathcal{A}}$ which is the unique solution of
\begin{equation}
\label{eqn:Case}
\begin{cases}
MA_{\omega}(u)=f\omega^{n}\\
u\in\mathcal{E}^{1}_{norm}(X,\omega,\psi).
\end{cases}
\end{equation}
In particular $u_{k}\to u$ in capacity.
\end{reptheorem}
\begin{proof}
We first observe that the existence of the unique solutions of $(\ref{eqn:K-case})$ follows by Theorem $A$ in \cite{DDNL18b}. \\
Moreover letting $u$ any weak accumulation point for $\{u_{k}\}_{k\in\mathbbm{N}}$ (there exists at least one by compactness), Lemma $2.8$ in \cite{DDNL18b} yields $MA_{\omega}(u)\geq f\omega^{n}$ and by the convergence of $f_{k}$ to $f$ we also obtain $\int_{X}f\omega^{n}=\lim_{k\to \infty} V_{\psi_{k}}$. Moreover since $u_{k}\leq \psi_{k}$ for any $k\in\mathbbm{N}$, by \cite{WN17} we obtain $\int_{X}MA_{\omega}(u)\leq \lim_{k\to \infty}V_{\psi_{k}}$. Hence $MA_{\omega}(u)=f\omega^{n}$ which in particular means that there is a unique weak accumulation point for $\{u_{k}\}_{k\in\mathbbm{N}}$ and that $\psi_{k}\to \psi$ as $k\to \infty$ since $V_{\psi_{k}}\to V_{\psi}$ (by Lemma \ref{lem:HomeoV}). Then it easily follows combining Fatou's Lemma with Proposition \ref{prop:PropProie} and Lemma \ref{lem:KeyConv} that for any $\varphi\in\mathcal{H}_{\omega}$ 
\begin{multline}
\label{eqn:P1}
\liminf_{k\to \infty} E_{\psi_{k}}^{*}\big(MA_{\omega}(u_{k})/V_{\psi_{k}}\big)\geq \liminf_{k\to \infty}\Big(E_{\psi_{k}}\big(P_{\omega}[\psi_{k}](\varphi)\big)+\int_{X}\big(\psi_{k}-P_{\omega}[\psi_{k}](\varphi)\big)f_{k}\omega^{n}\Big)\geq\\
\geq E_{\psi}\big(P_{\omega}[\psi](\varphi)\big)+\int_{X}\big(\psi-P_{\omega}[\psi](\varphi)\big)f\omega^{n}
\end{multline}
since $\big(\psi_{k}-P_{\omega}[\psi_{k}](\varphi)\big)f_{k}\to \big(\psi-P_{\omega}[\psi](\varphi)\big)f$ almost everywhere by Lemma \ref{lem:Referee}. Thus, for any $v\in\mathcal{E}^{1}(X,\omega,\psi)$ letting $\varphi_{j}\in\mathcal{H}_{\omega}$ be a decreasing sequence converging to $v$ (\cite{BK07}), from the inequality (\ref{eqn:P1}) we get
\begin{multline*}
\liminf_{k\to \infty} E_{\psi_{k}}^{*}\big(MA_{\omega}(u_{k})/V_{\psi_{k}}\big)\geq \limsup_{j\to\infty}\Big(E_{\psi}\big(P_{\omega}[\psi](\varphi_{j})\big)+\int_{X}\big(\psi-P_{\omega}[\psi](\varphi_{j})\big)f\omega^{n}\Big)=E_{\psi}(v)+\int_{X}(\psi-v)f\omega^{n}
\end{multline*}
using Proposition \ref{prop:PropertiesE} and the Monotone Converge Theorem. Hence by definition
\begin{equation}
\label{eqn:P4}
\liminf_{k\to \infty}E_{\psi_{k}}^{*}\big(MA_{\omega}(u_{k})/V_{\psi_{k}}\big)\geq E^{*}_{\psi}\big(f\omega^{n}/V_{\psi}\big).
\end{equation}
On the other hand since $||f_{k}||_{L^{p}},||f||_{L^{p}}$ are uniformly bounded where $p>1$ and $u_{k}\to u$, $\psi_{k}\to \psi$ in $L^{q}$ for any $q\in[1,+\infty)$ (see Theorem $1.48$ in \cite{GZ17}), we also have
$$
\int_{X}(\psi_{k}-u_{k})f_{k}\omega^{n}\to \int_{X}(\psi-u)f\omega^{n}<+\infty,
$$
which implies that $\int_{X}(\psi-u)MA_{\omega}(u)<+\infty$, i.e. $u\in\mathcal{E}^{1}(X,\omega,\psi)$ by Proposition \ref{prop:PropertiesE}. Moreover by Proposition \ref{prop:USC} we also get
$$
\limsup_{k\to \infty}E_{\psi_{k}}^{*}\big(MA_{\omega}(u_{k})/V_{\psi_{k}}\big)\leq E^{*}_{\psi}\big(MA_{\omega}(u)/V_{\psi}\big),
$$
which together with (\ref{eqn:P4}) leads to $MA_{\omega}(u_{k})\to MA_{\omega}(u)$ strongly in $Y_{\mathcal{A}}$ by definition (observe that $MA_{\omega}(u_{k})=f_{k}\omega^{n}\to MA_{\omega}(u)=f\omega^{n}$ weakly). Hence $u_{k}\to u$ strongly by Theorem \ref{thmB} while the convergence in capacity follows from Theorem \ref{thm:OldPropB}.
\end{proof}
\begin{rem}
\label{rem:Ele}
\emph{As said in the Introduction, the convergence in capacity of Theorem \ref{thmC} was already obtained in Theorem $1.4$ in \cite{DDNL19}. Indeed under the hypothesis of Theorem \ref{thmC} it follows from Lemma \ref{lem:KeyConv} and Lemma $3.4$ in \cite{DDNL19} that $d_{S}(\psi_{k},\psi)\to 0$ where $d_{S}$ is the pseudometric on $\{[u]\, : \, u\in PSH(X,\omega)\}$ introduced in \cite{DDNL19} where the class $[u]$ is given by the partial order $\preccurlyeq$.}
\end{rem}

{\small
\bibliographystyle{acm}
\bibliography{main}
}
\end{document}